\documentclass[review,3p,times,sort&compress]{elsarticle}
\usepackage{amssymb}
\usepackage{graphicx,fancyhdr}
\usepackage{multirow}
\usepackage{dsfont}
\usepackage{subfigure}
\usepackage{bbm}
\usepackage{booktabs}
\usepackage{amsthm}
\usepackage{bbm,bm}
\DeclareMathAlphabet{\mathpzc}{OT1}{pzc}{m}{it}
\usepackage{bbm}
\usepackage{dsfont}
\usepackage{yfonts}
\usepackage{amsmath}
\usepackage{amssymb}
\usepackage{amsfonts}
\usepackage{mathrsfs,float}
\usepackage{color}
\usepackage[colorlinks, citecolor=blue]{hyperref}

\makeatletter

\journal{}

\begin{document}
	
	\begin{frontmatter}
		
		
		\newtheorem{theorem}{Theorem}[section]
		\newtheorem{remark}[theorem]{Remark}
		\newtheorem{ex}{Example}
		\newtheorem{case}{Case}
		\newtheorem{pro}[theorem]{Proposition}
		\newtheorem{defi}[theorem]{Definition}
		\newtheorem{ass}{Assumption}
		\newtheorem{lemma}[theorem]{Lemma}
	\newtheorem{corollary}[theorem]{Corollary}
\newtheorem{assumption}[theorem]{Assumption}
		\newproof{pf}{Proof}
		\newproof{pot}{Proof of Theorem \ref{thm2}}
		\title{Local Legendre Frame Approximation from Equispaced Data}
		\author[1]{Benxue Gong}
	
		\author[1]{Zhenyu Zhao}
	\author[1]{Chenyang Wang}
		\address[1]{School of Mathematics and Statistics, Shandong University of Technology, Zibo, 255049, China}
\begin{abstract}
We propose a local Legendre frame (LLF) method for function approximation from equispaced data on a finite interval. Motivated by the difficulty of stable high-order polynomial approximation at equispaced points, especially in the presence of the Runge phenomenon, the method partitions the interval into subintervals, maps each subinterval to a common reference interval, and computes local coefficients by a truncated singular value decomposition (TSVD) regularization. Since all subintervals share the same local sampling matrix, the method admits a natural offline--online implementation.

We establish a quasi-optimal estimate for the regularized reconstruction and discuss practical parameter selection. Numerical results show that LLF attains high accuracy for relatively smooth and moderately oscillatory functions, while it remains applicable to highly oscillatory functions, although comparable accuracy generally requires more sampling points. For continuous piecewise smooth functions with derivative singularities, the method also provides an effective detect--localize--correct strategy based on one-sided coefficient-energy indicators. These results indicate that LLF provides a stable and flexible local approximation framework for equispaced data.
\end{abstract}
		
	\end{frontmatter}
	
\section{Introduction}

Function approximation from sampled data is a central topic in numerical analysis, with applications in scientific computing, inverse problems, and signal processing. A fundamental difficulty arises when approximating analytic functions from equispaced samples. The classical result of Platte, Trefethen, and Kuijlaars \cite{Platte2011} shows that any method achieving exponential convergence for such problems must necessarily be exponentially ill-conditioned. This impossibility result reflects the well-known instability of polynomial interpolation at equispaced nodes, commonly referred to as the Runge phenomenon.

Various approaches have been proposed to mitigate this difficulty \cite{AdcockPlatte2016,Boyd2009}. A particularly successful strategy is based on approximation over an extended domain combined with regularization. Within this framework, numerical frame approximation has emerged as a powerful paradigm \cite{AdcockHuybrechs2019,AdcockHuybrechs2020}. By employing redundant systems instead of bases, frame methods provide enhanced flexibility while allowing stability to be controlled through regularization.

A prominent example is the Fourier extension method \cite{Boyd2002,Huybrechs2010}, where a non-periodic function is approximated by a Fourier series defined on a larger interval. This approach effectively suppresses Gibbs-type artifacts and achieves rapid convergence when combined with suitable regularization \cite{AdcockHansen2012,AdcockHansenShadrin2014}. Extensions and efficient implementations of this idea have been extensively studied; see, for example, \cite{Bruno2010,BrunoPaul2022,Plonka2018}.

For problems posed on finite intervals, global approximation methods may suffer from high computational cost or limited adaptability. Domain decomposition techniques offer an effective alternative \cite{Gottlieb1989,Israeli1993,Israeli1994}. In particular, recent work on local Fourier extension methods \cite{Zhao2025local,Zhao2025weighted,Zhao2025fast} demonstrates that combining interval subdivision with reusable local discretizations leads to efficient offline--online implementations and improved adaptability to spatially varying features.

On the polynomial side, interval subdivision has long been used to alleviate the Runge phenomenon \cite{Boyd2009}. By reducing the size of subintervals, local approximation becomes more favorable. However, direct polynomial interpolation on equispaced nodes remains unstable when the local degree is large. This limitation is again consistent with the impossibility theorem \cite{Platte2011}, indicating that stability must be enforced through regularization rather than interpolation alone.

Polynomial frame approximation provides such a mechanism \cite{AdcockHuybrechs2019,AdcockShadrin2023}. By considering orthogonal polynomials on an extended interval and computing coefficients via regularized least squares, one obtains stable approximations from equispaced data. With appropriate oversampling and truncation, the method achieves rapid error decay down to a prescribed tolerance while maintaining controlled conditioning.

Motivated by these developments, we propose a \emph{local Legendre frame} (LLF) method for function approximation from equispaced samples. The main idea is to combine interval subdivision with polynomial frame approximation. The computational domain is partitioned into subintervals, each mapped to a common reference interval. On each subinterval, the function is approximated using Legendre polynomials defined on an extended domain, and the coefficients are computed via truncated singular value decomposition (TSVD). Since the same local discretization is used throughout, the associated sampling matrix and its TSVD can be precomputed and reused, leading to an efficient offline--online structure.

Numerical experiments show that the proposed LLF method achieves high accuracy for smooth and oscillatory functions using equispaced data, while maintaining numerical stability. Compared with local Fourier extension methods, LLF typically requires fewer effective degrees of freedom per subinterval, leading to slightly reduced online computational cost in certain regimes. In addition, the framework extends naturally to continuous piecewise smooth functions: by exploiting local coefficient indicators, singularity-containing subintervals can be detected and treated using one-sided reconstruction and interface correction.

The remainder of the paper is organized as follows. Section~2 introduces the construction of the local Legendre frame. Section~3 discusses the numerical implementation and parameter selection. Section~4 reports numerical experiments, including smooth, oscillatory, and piecewise smooth cases. Section~5 concludes the paper.
\section{Local Legendre frame approximation}
\label{sec:method}

\subsection{Partition and local rescaling}

Let
\[
I=[a,b], \qquad a=a_0<a_1<\cdots<a_K=b,
\]
and denote
\[
I_k=[a_{k-1},a_k], \qquad k=1,2,\dots,K.
\]
For a function \(f:I\to\mathbb{R}\) (or \(\mathbb{C}\)), write
\[
f_k(x)=f(x), \qquad x\in I_k.
\]
To transfer all local approximation problems to a common reference interval, we introduce the affine transformation
\[
x=c_k+s_k t, \qquad t\in \Lambda:=[-1,1],
\]
where
\[
c_k=\frac{a_{k-1}+a_k}{2}, \qquad s_k=\frac{a_k-a_{k-1}}{2}.
\]
The corresponding scaled local function is
\[
g_k(t)=f_k(c_k+s_k t), \qquad t\in \Lambda.
\]

Thus, the global approximation problem on \(I\) is reduced to a family of local approximation problems for the functions \(g_k\) on the fixed reference interval \(\Lambda\).

\subsection{Piecewise polynomial approximation as a comparison tool}

We first recall the approximation effect of interval subdivision in the classical interpolation setting. This will be used later only as a comparison tool in the quasi-optimal estimate for the local frame method.

Fix an integer \(n\ge 1\), and let
\[
t_j=-1+jh, \qquad h=\frac{2}{n}, \qquad j=0,1,\dots,n,
\]
be the equispaced nodes on \(\Lambda\). For each \(k\), let \(L_n g_k\) denote the degree-\(n\) Lagrange interpolation polynomial of \(g_k\) at these nodes, and define the corresponding piecewise polynomial approximation of \(f\) by
\[
(L_n f)(x):=(L_n g_k)\!\left(\frac{x-c_k}{s_k}\right), \qquad x\in I_k.
\]

The following lemma quantifies the approximation effect of interval subdivision.

\begin{lemma}[Piecewise interpolation error]
\label{lem:piecewise-interpolation}
Assume that \(f\in C^{n+1}[a,b]\) and
\[
\|f^{(n+1)}\|_{L^\infty(a,b)}\le C_f.
\]
Let
\[
h_I=\max_{1\le k\le K}(a_k-a_{k-1})
\]
be the maximal subinterval length. Then, for all \(x\in[a,b]\),
\[
|f(x)-L_n f(x)|\le \frac{h_I^{\,n+1}}{n^{\,n+1}}\,C_f.
\]
\end{lemma}

\begin{proof}
For each \(k\), the interpolation remainder for \(g_k\) on \(\Lambda\) is
\[
g_k(t)-(L_n g_k)(t)
=
\frac{g_k^{(n+1)}(\xi_t)}{(n+1)!}\,\omega_{n+1}(t),
\qquad
\omega_{n+1}(t)=\prod_{j=0}^{n}(t-t_j),
\]
for some \(\xi_t\in\Lambda\). Since \(h=2/n\), we have
\[
|\omega_{n+1}(t)|
\le
\prod_{j=0}^{n}(j+1)h
=
(n+1)!\,h^{n+1}.
\]
Moreover, by the chain rule,
\[
g_k^{(n+1)}(t)=s_k^{\,n+1}f_k^{(n+1)}(c_k+s_k t),
\]
and hence
\[
\|g_k^{(n+1)}\|_{L^\infty(\Lambda)}
\le s_k^{\,n+1} C_f.
\]
Therefore,
\[
|g_k(t)-(L_n g_k)(t)|
\le s_k^{\,n+1} C_f\, h^{n+1}
\le \left(\frac{h_I}{2}\right)^{n+1} C_f \left(\frac{2}{n}\right)^{n+1}
=
\frac{h_I^{\,n+1}}{n^{\,n+1}}\,C_f.
\]
Transforming back from \(t\) to \(x\) yields the result.
\end{proof}

Lemma~\ref{lem:piecewise-interpolation} shows that, for fixed polynomial degree \(n\), the interpolation error decreases as the maximal subinterval length \(h_I\) decreases. This explains why interval subdivision improves local approximability. At the same time, the lemma concerns only approximation error in exact arithmetic. If one attempts to increase the local degree in direct equispaced interpolation, the corresponding interpolation process still suffers from the growth of the Lebesgue constant and hence from increasing numerical instability. This is precisely the point at which the frame formulation becomes useful.

\subsection{Construction of the local Legendre frame}

To obtain a numerically stable local approximation procedure, we replace direct local interpolation by a polynomial frame approximation on an extended interval. Let
\[
\widetilde{\Lambda}=[-T,T], \qquad T>1.
\]
Let \(\{p_\ell\}_{\ell\ge0}\) be the orthonormal Legendre basis on \([-1,1]\), normalized by
\[
\int_{-1}^1 p_\ell(y)p_j(y)\,dy=\delta_{\ell j},
\qquad \ell,j\ge0.
\]
We define the scaled Legendre system on \(\widetilde{\Lambda}\) by
\[
P_\ell^{(T)}(t)=\frac{1}{\sqrt{T}}\,p_\ell\!\left(\frac{t}{T}\right),
\qquad t\in[-T,T], \qquad \ell\ge0.
\]
A direct change of variables shows that
\[
\int_{-T}^T P_\ell^{(T)}(t)P_j^{(T)}(t)\,dt=\delta_{\ell j},
\]
so that \(\{P_\ell^{(T)}\}_{\ell\ge0}\) forms an orthonormal basis of \(L^2([-T,T])\).

For a fixed integer \(N\ge0\), define
\[
V_N^{(T)}=\mathrm{span}\{P_0^{(T)},P_1^{(T)},\dots,P_N^{(T)}\}.
\]
Restricting these functions from \([-T,T]\) to the reference interval \(\Lambda=[-1,1]\), we obtain the finite system
\[
\Phi_N^{(T)}=\{P_\ell^{(T)}|_{\Lambda}\}_{\ell=0}^N.
\]
Since the orthogonality is inherited from the larger interval \([-T,T]\) rather than from \(\Lambda\), the restricted system is no longer orthogonal on \(\Lambda\). Instead, it provides a redundant polynomial representation on \(\Lambda\), which we regard as a local Legendre frame.

We now introduce the associated synthesis operator
\[
\mathcal L_N^{(T)}:\mathbb C^{N+1}\to L^2(\Lambda),
\qquad
\mathcal L_N^{(T)}(\mathbf c)
=
\sum_{\ell=0}^{N} c_\ell P_\ell^{(T)}|_{\Lambda},
\]
where \(\mathbf c=(c_0,\dots,c_N)^T\in\mathbb C^{N+1}\).

Given a regularization parameter \(\epsilon>0\), let
\[
Q_{N,\epsilon}^{(T)}:L^2(\Lambda)\to V_N^{(T)}|_{\Lambda}
\]
denote the TSVD-regularized approximation operator associated with \(\mathcal L_N^{(T)}\). For each local function \(g_k\), we define its local Legendre frame approximation by
\[
Q_{N,\epsilon}^{(T)} g_k \in V_N^{(T)}|_{\Lambda},
\]
and then assemble the global approximation piecewise as
\[
(\mathcal P_{N,K}^{T,\epsilon}f)(x)
=
(Q_{N,\epsilon}^{(T)} g_k)\!\left(\frac{x-c_k}{s_k}\right),
\qquad x\in I_k.
\]

\subsection{A quasi-optimal estimate}

The following theorem gives the basic approximation estimate for the regularized local frame operator.

\begin{theorem}[Local quasi-optimality]
\label{thm:local-quasioptimal}
For each \(g\in L^2(\Lambda)\),
\[
\|g-Q_{N,\epsilon}^{(T)} g\|_{L^2(\Lambda)}
\le
\inf_{\mathbf c\in\mathbb C^{N+1}}
\left\{
\left\|g-\mathcal L_N^{(T)}(\mathbf c)\right\|_{L^2(\Lambda)}
+
\epsilon \|\mathbf c\|_2
\right\}.
\]
Consequently,
\[
\|f-\mathcal P_{N,K}^{T,\epsilon}f\|_{L^2(I)}^2
=
\sum_{k=1}^{K}
s_k\,\|g_k-Q_{N,\epsilon}^{(T)} g_k\|_{L^2(\Lambda)}^2.
\]
\end{theorem}

\begin{proof}
The first inequality is the standard quasi-optimal estimate for truncated-SVD frame approximation. The second identity follows from the change of variables
\(
x=c_k+s_k t
\)
on each subinterval \(I_k\).
\end{proof}

To connect this result with Lemma~\ref{lem:piecewise-interpolation}, we take the interpolation degree equal to the frame degree, namely \(n=N\), and set
\[
p_k=L_N g_k.
\]
Since \(p_k\in\mathbb P_N\subset V_N^{(T)}\), there exists \(\mathbf c_k^{\mathrm{int}}\in\mathbb C^{N+1}\) such that
\[
p_k=\mathcal L_N^{(T)}(\mathbf c_k^{\mathrm{int}}).
\]
Applying Theorem~\ref{thm:local-quasioptimal} with \(\mathbf c=\mathbf c_k^{\mathrm{int}}\) gives
\[
\|g_k-Q_{N,\epsilon}^{(T)} g_k\|_{L^2(\Lambda)}
\le
\|g_k-p_k\|_{L^2(\Lambda)}
+
\epsilon \|\mathbf c_k^{\mathrm{int}}\|_2.
\]
Thus the local frame approximation error is controlled by the interpolation error together with a regularization term involving the representing coefficient vector.

This yields the following corollary.

\begin{corollary}
\label{cor:interpolation-comparison}
Assume the hypotheses of Lemma~\ref{lem:piecewise-interpolation} with \(n=N\), and let
\[
p_k=L_N g_k.
\]
Then, for each \(k=1,\dots,K\), there exists \(\mathbf c_k^{\mathrm{int}}\in\mathbb C^{N+1}\) such that
\[
p_k=\mathcal L_N^{(T)}(\mathbf c_k^{\mathrm{int}}),
\]
and
\[
\|g_k-Q_{N,\epsilon}^{(T)} g_k\|_{L^2(\Lambda)}
\le
\|g_k-p_k\|_{L^2(\Lambda)}
+
\epsilon \|\mathbf c_k^{\mathrm{int}}\|_2.
\]
Consequently,
\[
\|g_k-Q_{N,\epsilon}^{(T)} g_k\|_{L^2(\Lambda)}
\le
\sqrt{2}\,\frac{h_I^{\,N+1}}{N^{\,N+1}}\,C_f
+
\epsilon \|\mathbf c_k^{\mathrm{int}}\|_2,
\]
and hence
\[
\|f-\mathcal P_{N,K}^{T,\epsilon}f\|_{L^2(I)}^2
\le
\sum_{k=1}^K
s_k
\left(
\sqrt{2}\,\frac{h_I^{\,N+1}}{N^{\,N+1}}\,C_f
+
\epsilon \|\mathbf c_k^{\mathrm{int}}\|_2
\right)^2.
\]
\end{corollary}

\begin{proof}
The representation
\[
p_k=\mathcal L_N^{(T)}(\mathbf c_k^{\mathrm{int}})
\]
follows from \(p_k\in\mathbb P_N\subset V_N^{(T)}\). The first inequality is exactly the estimate above. Moreover,
\[
\|g_k-p_k\|_{L^2(\Lambda)}
\le
\sqrt{2}\,\|g_k-p_k\|_{L^\infty(\Lambda)},
\]
and Lemma~\ref{lem:piecewise-interpolation} gives
\[
\|g_k-p_k\|_{L^2(\Lambda)}
\le
\sqrt{2}\,\frac{h_I^{\,N+1}}{N^{\,N+1}}\,C_f.
\]
Substituting this into the previous bound yields the desired estimates.
\end{proof}

Corollary~\ref{cor:interpolation-comparison} shows that the interval partition improves local approximability through the factor \(h_I^{\,N+1}\), while the frame formulation introduces only the additional regularization term \(\epsilon\|\mathbf c_k^{\mathrm{int}}\|_2\). In this way, approximation quality and numerical stabilization are cleanly separated.

\begin{remark}
The interpolation polynomial is used only as a comparison function in the quasi-optimal estimate; it is not part of the computational procedure. This is precisely the advantage of the frame formulation: it retains the approximation power of local polynomials without relying on the numerical stability of direct equispaced interpolation.
\end{remark}
\section{Numerical implementation and parameter selection}
\label{sec:implementation}

In this section, we present the discrete implementation of the local Legendre frame (LLF) method and discuss the practical role of its main numerical parameters. After mapping each physical subinterval to the common reference interval \(\Lambda=[-1,1]\), all local approximation problems share the same discrete structure. Hence the sampling matrix and its truncated singular value decomposition (TSVD) can be precomputed once and reused on all subintervals.
\subsection{Discrete local approximation on the reference interval}

Let \(N\ge 0\) be the local polynomial degree, \(T\ge 1\) the extension parameter, and \(m\ge N+1\) the number of equispaced sample points on the reference interval \(\Lambda=[-1,1]\). In practice, we often write
\[
m=\lceil \gamma (N+1)\rceil,\qquad \gamma\ge 1,
\]
where \(\gamma\) denotes the oversampling ratio.

We define the reference nodes by
\[
t_j=-1+\frac{2j}{m-1},\qquad j=0,1,\dots,m-1.
\]
For a given subinterval
\[
I_k=[a_{k-1},a_k],\qquad
c_k=\frac{a_{k-1}+a_k}{2},\qquad
s_k=\frac{a_k-a_{k-1}}{2},
\]
the corresponding physical nodes are
\[
x_{k,j}=c_k+s_k t_j,\qquad j=0,1,\dots,m-1.
\]
The sampled local data are collected in the vector
\[
g_k=
\bigl(g_k(t_0),g_k(t_1),\dots,g_k(t_{m-1})\bigr)^T
=
\bigl(f(x_{k,0}),f(x_{k,1}),\dots,f(x_{k,m-1})\bigr)^T.
\]

Recall that the scaled Legendre frame on the extended interval \([-T,T]\) is given by
\[
P_\ell^{(T)}(t)=\frac{1}{\sqrt{T}}\,p_\ell\!\left(\frac{t}{T}\right),\qquad \ell\ge 0,
\]
where \(\{p_\ell\}_{\ell\ge0}\) denotes the orthonormal Legendre basis on \([-1,1]\). Restricting these functions to \(\Lambda\), we form the discrete sampling matrix
\[
A_{m,N}^{(T)}\in\mathbb{R}^{m\times (N+1)},\qquad
\bigl(A_{m,N}^{(T)}\bigr)_{j,\ell}
=
\frac{1}{\sqrt m}P_\ell^{(T)}(t_j),
\]
for \(j=0,\dots,m-1\) and \(\ell=0,\dots,N\).

The coefficient vector \(c_k\in\mathbb{R}^{N+1}\) of the local approximation
\[
q_k(t)=\sum_{\ell=0}^N c_{k,\ell}P_\ell^{(T)}(t)
\]
is determined from the least-squares system
\[
A_{m,N}^{(T)}c_k\approx \frac{1}{\sqrt m}g_k.
\]
Since the same reference nodes and the same local frame are used on every subinterval, the matrix \(A_{m,N}^{(T)}\) is independent of \(k\). This makes the method naturally suited to an offline--online implementation.

\subsection{TSVD regularization and reusable factorization}

As in polynomial frame approximation, the matrix \(A_{m,N}^{(T)}\) becomes increasingly ill-conditioned as \(N\) grows. Therefore, the coefficients are computed by truncated singular value decomposition. Let
\[
A_{m,N}^{(T)} = U\Sigma V^T,
\]
where
\[
\Sigma=\mathrm{diag}(\sigma_0,\sigma_1,\dots,\sigma_r),\qquad
\sigma_0\ge \sigma_1\ge \cdots \ge \sigma_r>0,
\]
and \(r=\min\{m-1,N\}\). For a prescribed truncation threshold \(\epsilon>0\), define
\[
\mathcal I_\epsilon=\{j:\sigma_j>\epsilon\}.
\]
Let \(U_\epsilon\) and \(V_\epsilon\) denote the matrices formed by the singular vectors corresponding to \(\mathcal I_\epsilon\), and let
\[
\Sigma_\epsilon^{-1}=\mathrm{diag}(\sigma_j^{-1})_{j\in\mathcal I_\epsilon}.
\]

Then the TSVD coefficient vector on the \(k\)-th subinterval is
\[
c_k^\epsilon
=
V_\epsilon \Sigma_\epsilon^{-1} U_\epsilon^T
\left(\frac{1}{\sqrt m}g_k\right).
\]
The corresponding local approximation is
\[
q_k^\epsilon(t)=\sum_{\ell=0}^N c_{k,\ell}^\epsilon P_\ell^{(T)}(t),\qquad t\in[-1,1],
\]
and the global approximation is assembled piecewise as
\[
\bigl(P_{N,K}^{T,\epsilon}f\bigr)(x)
=
q_k^\epsilon\!\left(\frac{x-c_k}{s_k}\right),\qquad x\in I_k.
\]

The implementation therefore has a natural offline--online structure.

\medskip
\noindent
\textbf{Offline stage.}
For prescribed \(T\), \(N\), \(m\), and \(\epsilon\),
\begin{enumerate}
\item construct the reference nodes \(\{t_j\}_{j=0}^{m-1}\);
\item evaluate the frame functions \(P_\ell^{(T)}(t_j)\);
\item assemble the matrix \(A_{m,N}^{(T)}\);
\item compute and store its singular value decomposition
\[
A_{m,N}^{(T)}=U\Sigma V^T;
\]
\item determine the retained singular modes \(\mathcal I_\epsilon\) and store the truncated factors
\[
U_\epsilon,\qquad \Sigma_\epsilon^{-1},\qquad V_\epsilon.
\]
\end{enumerate}

\medskip
\noindent
\textbf{Online stage.}
For each subinterval \(I_k\),
\begin{enumerate}
\item sample the local data \(g_k\) at the physical nodes \(\{x_{k,j}\}_{j=0}^{m-1}\);
\item compute
\[
\alpha_k
=
U_\epsilon^T\left(\frac{1}{\sqrt m}g_k\right);
\]
\item compute
\[
\beta_k=\Sigma_\epsilon^{-1}\alpha_k;
\]
\item compute
\[
c_k^\epsilon=V_\epsilon\beta_k;
\]
\item evaluate \(q_k^\epsilon\) and assemble the global approximation.
\end{enumerate}

\begin{remark}
Although the TSVD solution can be written formally as
\[
c_k^\epsilon
=
V\Sigma_\epsilon^\dagger U^T
\left(\frac{1}{\sqrt m}g_k\right),
\]
we do \emph{not} explicitly form the dense product \(V\Sigma_\epsilon^\dagger U^T\) in the offline stage. The reason is numerical rather than algebraic: the projected coefficients
\[
U_\epsilon^T\left(\frac{1}{\sqrt m}g_k\right)
\]
typically decay with the singular modes, and only after this projection is it numerically safe to apply the amplification by \(\Sigma_\epsilon^{-1}\). If one precomputes the full product \(V\Sigma_\epsilon^\dagger U^T\), the large reciprocals of small singular values are introduced too early, which may noticeably deteriorate finite-precision accuracy. For this reason, the online computation is carried out in the ordered form
\[
U_\epsilon^T g_k
\;\rightarrow\;
\Sigma_\epsilon^{-1}(U_\epsilon^T g_k)
\;\rightarrow\;
V_\epsilon \Sigma_\epsilon^{-1} U_\epsilon^T g_k.
\]
\end{remark}

\subsection{Choice of the extension parameter \(T\)}

We next examine the role of the extension parameter \(T\). As in the preceding discussion, we consider the oscillatory test function
\[
f(x)=\sin(\omega x),\qquad x\in[-1,1],
\]
and measure the approximation error on a refined grid. Throughout this subsection, the truncation threshold is fixed at
\[
\epsilon=10^{-14}.
\]

In practice, due to the accumulation of rounding errors, it is difficult to consistently reach the level \(10^{-14}\) for general functions. Therefore, in determining the admissible range of \(T\), we adopt the more realistic tolerance
\[
5\times 10^{-13}
\]
as the target accuracy.

Figure~\ref{fig:Ttest} displays the approximation error as a function of \(T\) under different choices of \(N\), \(K\), \(\omega\), and \(\gamma\). The numerical results reveal a clear stable operating interval
\[
[T_1,T_2],
\]
within which the approximation error remains close to the target tolerance, while outside this interval the error increases rapidly.

The first observation is that the lower threshold \(T_1\) is governed mainly by the oversampling ratio \(\gamma\). Table~\ref{tab:T1} lists the numerically observed values of \(T_1\) required to reach the accuracy level \(5\times10^{-13}\) for several choices of \(\gamma\). In particular, for the most economical choice
\[
\gamma=1,
\]
we observe that
\[
T_1\approx 5.6.
\]
Since the threshold depends mildly on the target tolerance and finite-precision effects, it is natural in practice to choose \(T\) slightly above this value. Accordingly, we adopt
\[
T=6
\]
in subsequent computations.
\begin{figure}[htbp]
\begin{center}
\subfigure[$\omega=40$, $K=4$, $N=150$]{
\resizebox*{6.5cm}{!}{\includegraphics{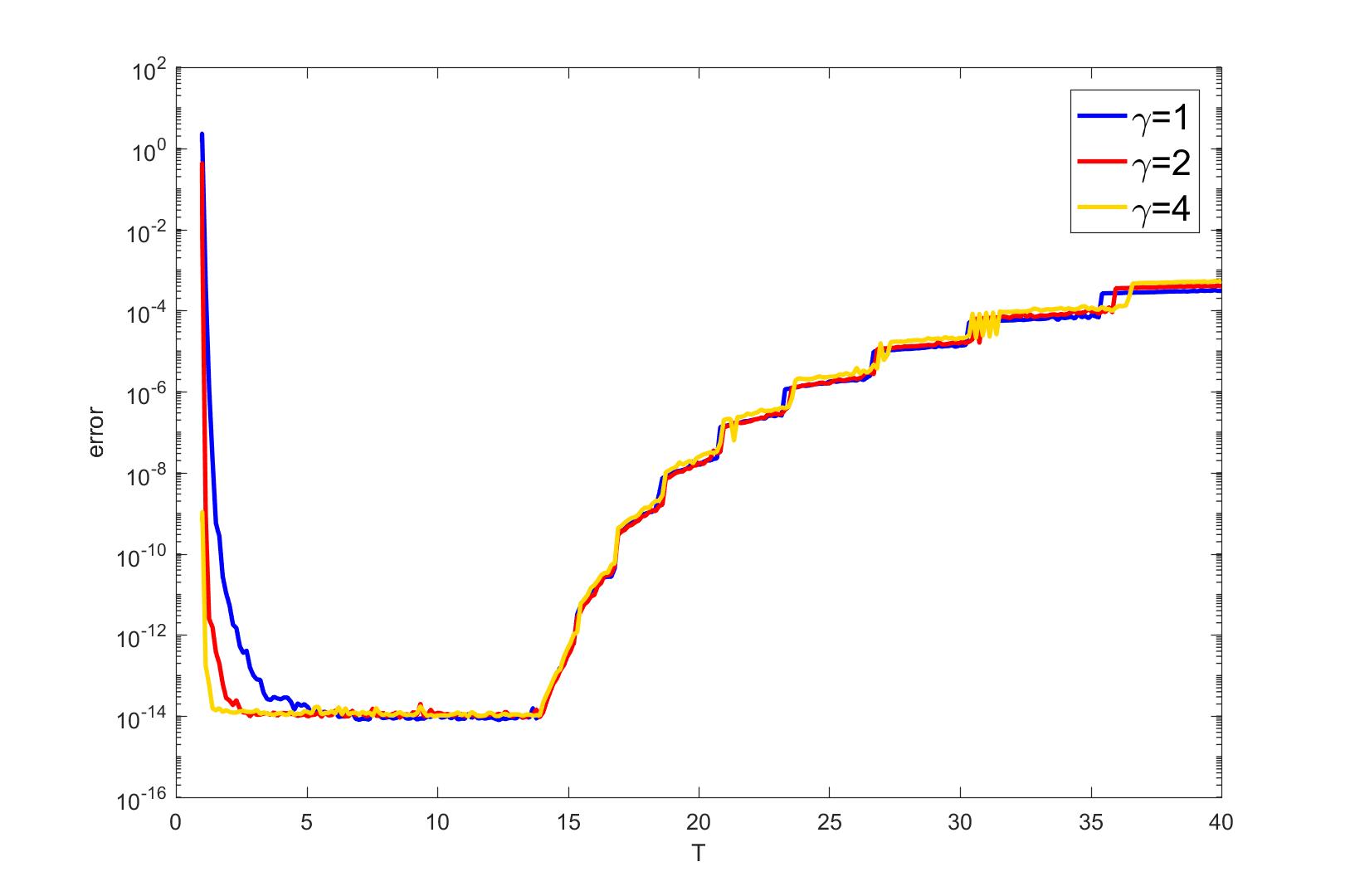}}
}%
\subfigure[$\omega=40$, $K=4$, $\gamma=1$]{
\resizebox*{6.5cm}{!}{\includegraphics{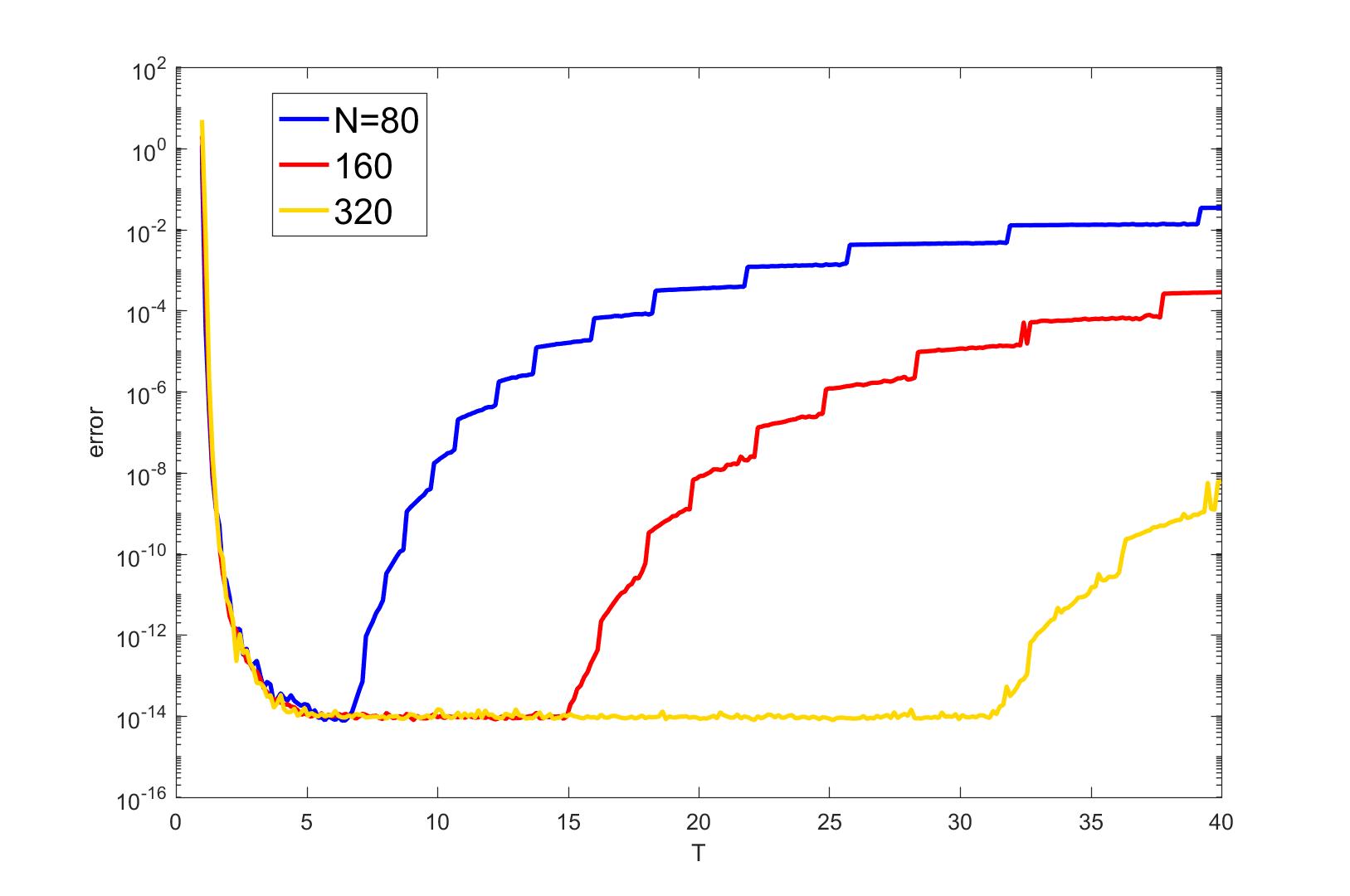}}
}%

\subfigure[$\omega=40$, $\gamma=1$, $N=150$]{
\resizebox*{6.5cm}{!}{\includegraphics{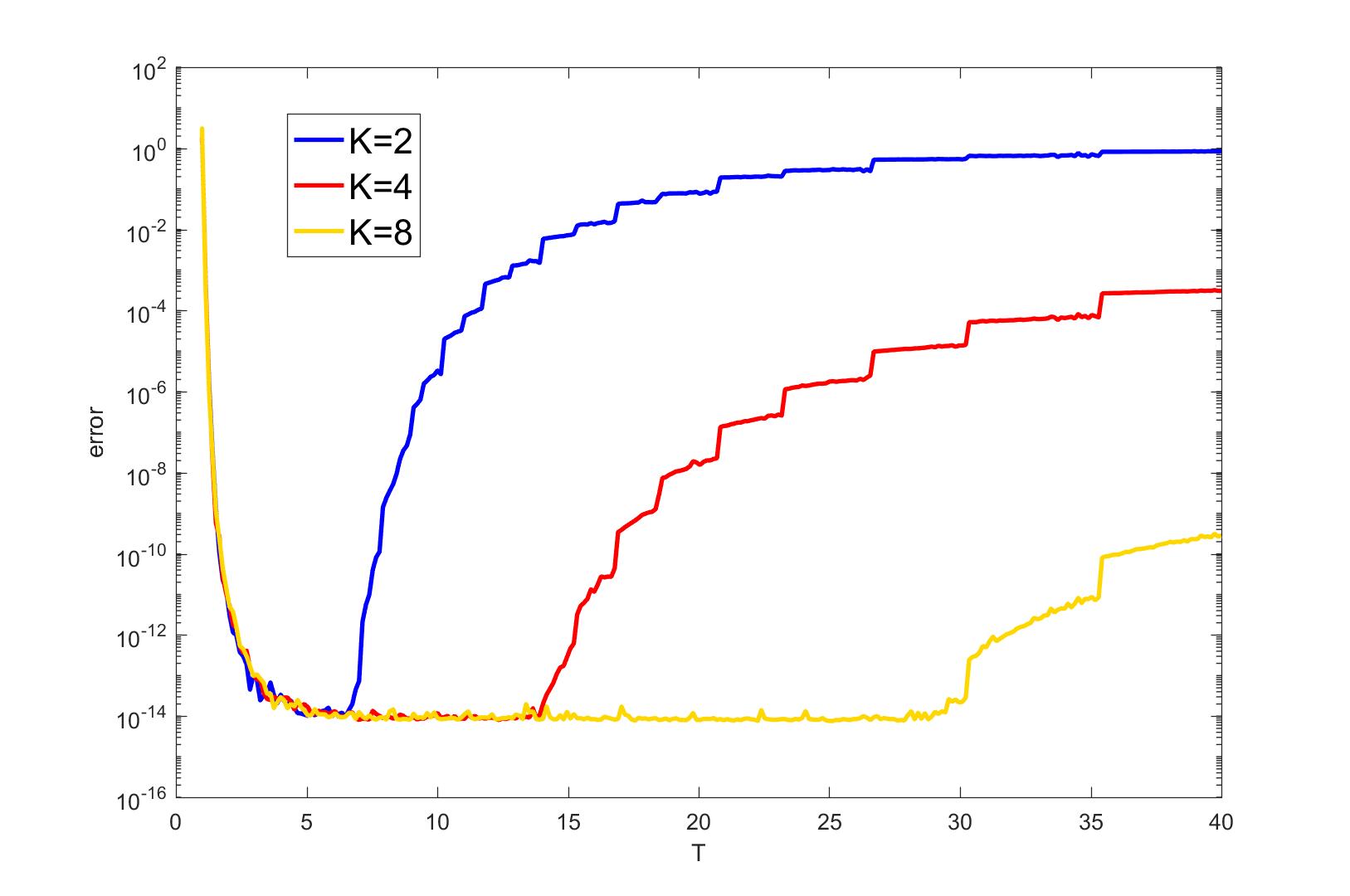}}
}%
\subfigure[$\gamma=1$, $K=4$, $N=150$]{
\resizebox*{6.5cm}{!}{\includegraphics{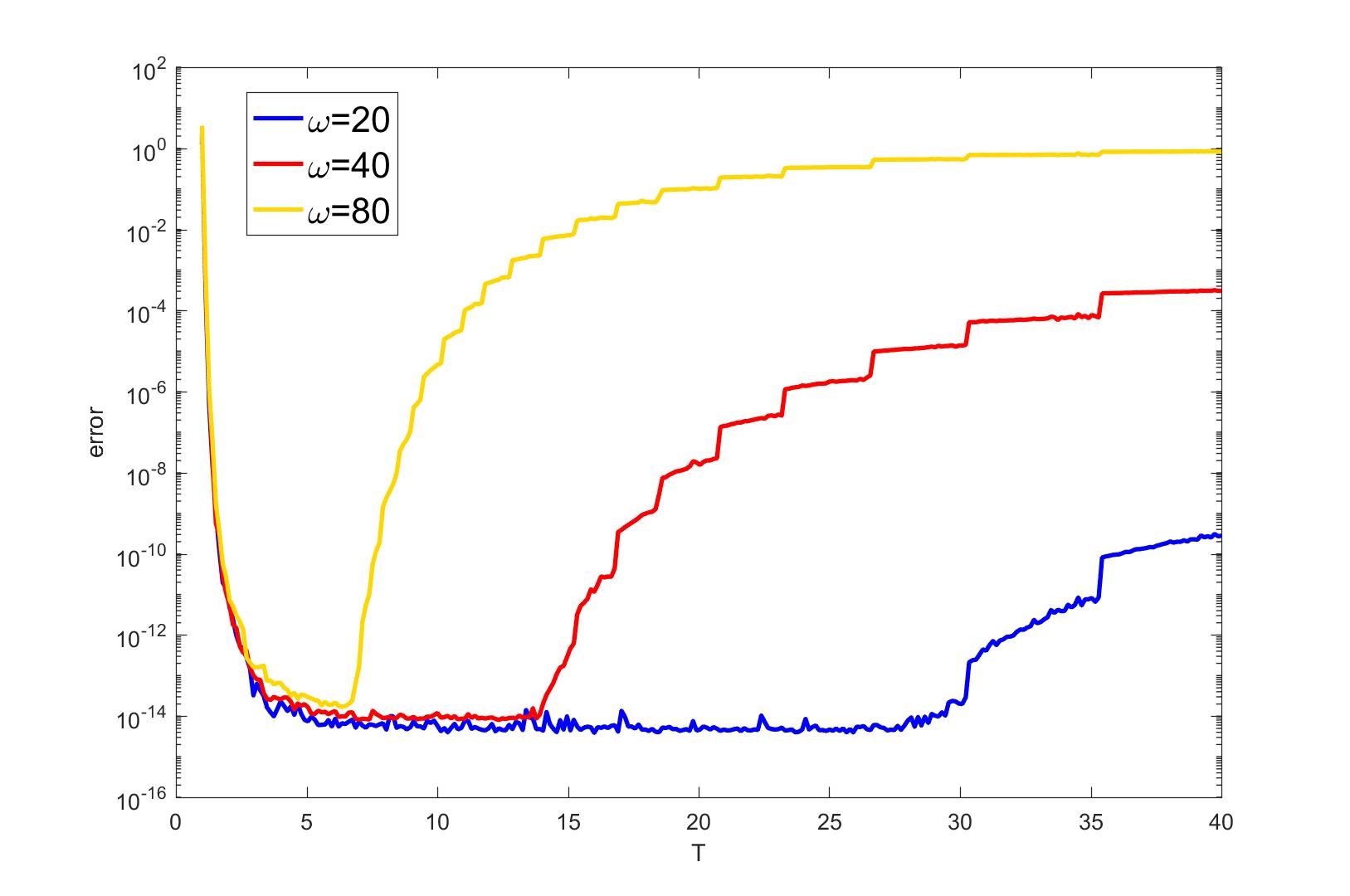}}
}%

\caption{Approximation error as a function of the extension parameter \(T\).
Panel (a) varies the oversampling ratio \(\gamma\), panel (b) varies the local degree \(N\), panel (c) varies the number of subintervals \(K\), and panel (d) varies the frequency \(\omega\). The results show that the admissible range of \(T\) is governed mainly by \(\gamma\), while the upper bound \(T_2\) depends on the interplay between \(N\), \(K\), and \(\omega\). Once \(T\) lies in the stable range, its influence on the approximation error becomes mild.}
\label{fig:Ttest}
\end{center}
\end{figure}

The second observation concerns the upper threshold \(T_2\). From Figure~\ref{fig:Ttest}, one can see that \(T_2\) increases with the local degree \(N\) and the number of subintervals \(K\), and decreases with the oscillation frequency \(\omega\). Numerically, we observe that
\[
T_2 \approx \frac{N K}{\omega}.
\]
This behavior can be understood heuristically as follows. Increasing \(T\) enlarges the effective extension domain, which reduces the resolution of the restricted frame on \([-1,1]\). When \(T\) becomes too large relative to the effective local oscillatory complexity (measured by \(\omega\)) and the available degrees of freedom (measured by \(N\) and \(K\)), the approximation loses its ability to resolve the oscillations accurately, leading to a rapid growth of the error. Thus, the ratio \(NK/\omega\) provides a natural upper bound for admissible values of \(T\).

Finally, we note that, once \(T\) lies within the stable range \([T_1,T_2]\), further variation of \(T\) has little influence on the degree \(N\) required to attain the target accuracy; see also Figure~\ref{fig:Ntest}(b). This indicates that \(T\) acts primarily as a stability parameter rather than as the main resolution parameter. Since \(T\) does not materially affect the online cost, it is reasonable to fix it slightly above the lower edge of the admissible interval. For this reason, we take
\[
T=6
\]
throughout the remainder of the paper.

\begin{table}[htbp]
\begin{center}
\caption{Numerically observed values of \(T_1\) required to reach the accuracy level \(5\times10^{-13}\) for \(f(x)=\sin(\omega x)\) on \([-1,1]\).}
\label{tab:T1}
\small
{\begin{tabular*}{\textwidth}{@{\extracolsep\fill}cccccccccccccc} \toprule
&$\gamma=1$&$\gamma=2$&$\gamma=3$&$\gamma=4$&$\gamma=5$\\\midrule
$T_1$ &5.6&2.6&1.8&1.6&1.3\\
\bottomrule
\end{tabular*}}
\end{center}
\end{table}

\subsection{Dependence on the local degree \(N\)}

We now study the role of the local degree \(N\). As in the previous subsection, we assess accuracy with respect to the practical tolerance
\[
5\times 10^{-13},
\]
which better reflects attainable accuracy in finite precision.

The numerical results show that the value of \(N\) required to reach this accuracy level is governed primarily by the local oscillatory complexity, and is essentially independent of the oversampling ratio \(\gamma\); see Figure~\ref{fig:Ntest}(b). By contrast, increasing \(\gamma\) enlarges the number of sampling points
\[
m=\lceil \gamma (N+1)\rceil,
\]
and hence increases the online computational cost without reducing the degree threshold in any essential way.

To quantify this dependence, we introduce the local frequency \(\omega_\Delta\) on each subinterval. For a function with global frequency \(\omega\), subdivision reduces the effective local oscillation, so that the required local degree is determined by \(\omega_\Delta\), rather than by \(\omega\) itself. This is consistent with the general mechanism observed in local Fourier extension: interval subdivision lowers the local oscillatory complexity and thereby reduces the approximation difficulty.

Table~\ref{tab:NCw} reports the numerically observed values of \(m\) and \(C_\Delta\) for different local frequencies \(\omega_\Delta\) in the case \(\gamma=1\). Here \(C_\Delta\) denotes the number of retained singular values after TSVD truncation. Since the total number of nodes on the whole interval is
\[
K(m-1)+1,
\]
and the online computational cost is proportional to
\[
C_\Delta m K,
\]
a natural indicator for comparing different parameter choices is
\[
\frac{mC_\Delta}{\omega_\Delta}.
\]

The last row of Table~\ref{tab:NCw} shows that this quantity remains within a moderate range once \(\omega_\Delta\) is not too small. In particular, the cases
\[
N=15\quad (\omega_\Delta=1),\qquad N=39\quad (\omega_\Delta=4)
\]
lead to comparable values of \(mC_\Delta/\omega_\Delta\). Their practical implications, however, are different. For low-frequency functions, the required total number of nodes is often well below \(40\), so choosing a large value such as \(N=39\) would be unnecessary. Hence, for low-to-moderate local frequencies, a natural default choice is
\[
\gamma=1,\qquad N=15,\qquad m=15.
\]

For highly oscillatory functions, the situation changes. Although \(N=15\) and \(N=39\) lead to comparable cost indicators at the target accuracy, the ratio
\[
\frac{N}{\omega_\Delta}
\]
is smaller when \(N=39\). This implies that, for a given global frequency, fewer subintervals are required, and hence the total number of nodes on the whole interval may be reduced. Therefore, for high-frequency problems it is advantageous to use a larger local degree, and a practical choice is
\[
\gamma=1,\qquad N=m=40.
\]

We now combine these observations with the discussion in Section~3.3. The admissible range of \(T\) is controlled by the oversampling ratio \(\gamma\), and a smaller value of \(T\) is desirable since it enlarges the effective resolution ratio \(NK/T\), allowing each subinterval to represent higher-frequency components. However, decreasing \(T\) requires a larger \(\gamma\) to maintain stability, as reflected by the dependence of \(T_1\) on \(\gamma\). Consequently, the relevant quantity is not \(T_1\) or \(\gamma\) alone, but their product.

From Table~\ref{tab:T1}, we observe that the product \(T_1\gamma\) remains essentially comparable for \(\gamma=1\) and \(\gamma=2\), while it becomes larger for higher values of \(\gamma\). This indicates that increasing \(\gamma\) does not provide a net advantage when both stability and resolution are taken into account. Therefore, from the viewpoint of overall efficiency, it is natural to adopt the simplest choice
\[
\gamma=1,
\]
and to control resolution primarily through \(N\) and the subdivision parameter \(K\).

These observations clarify the distinct roles of the main parameters:
\begin{itemize}
\item \(T\) is chosen mainly for stability;
\item \(N\) controls local resolution;
\item increasing \(K\) reduces the required local degree by decreasing the effective local frequency \(\omega_\Delta\);
\item \(\gamma\) mainly controls the sampling redundancy and thus affects the computational cost, while its interaction with \(T\) is reflected through the product \(T_1\gamma\).
\end{itemize}

\begin{figure}[htbp]
\begin{center}
\subfigure[$\omega=60$, $K=4$, $N=120$]{
\resizebox*{6cm}{!}{\includegraphics{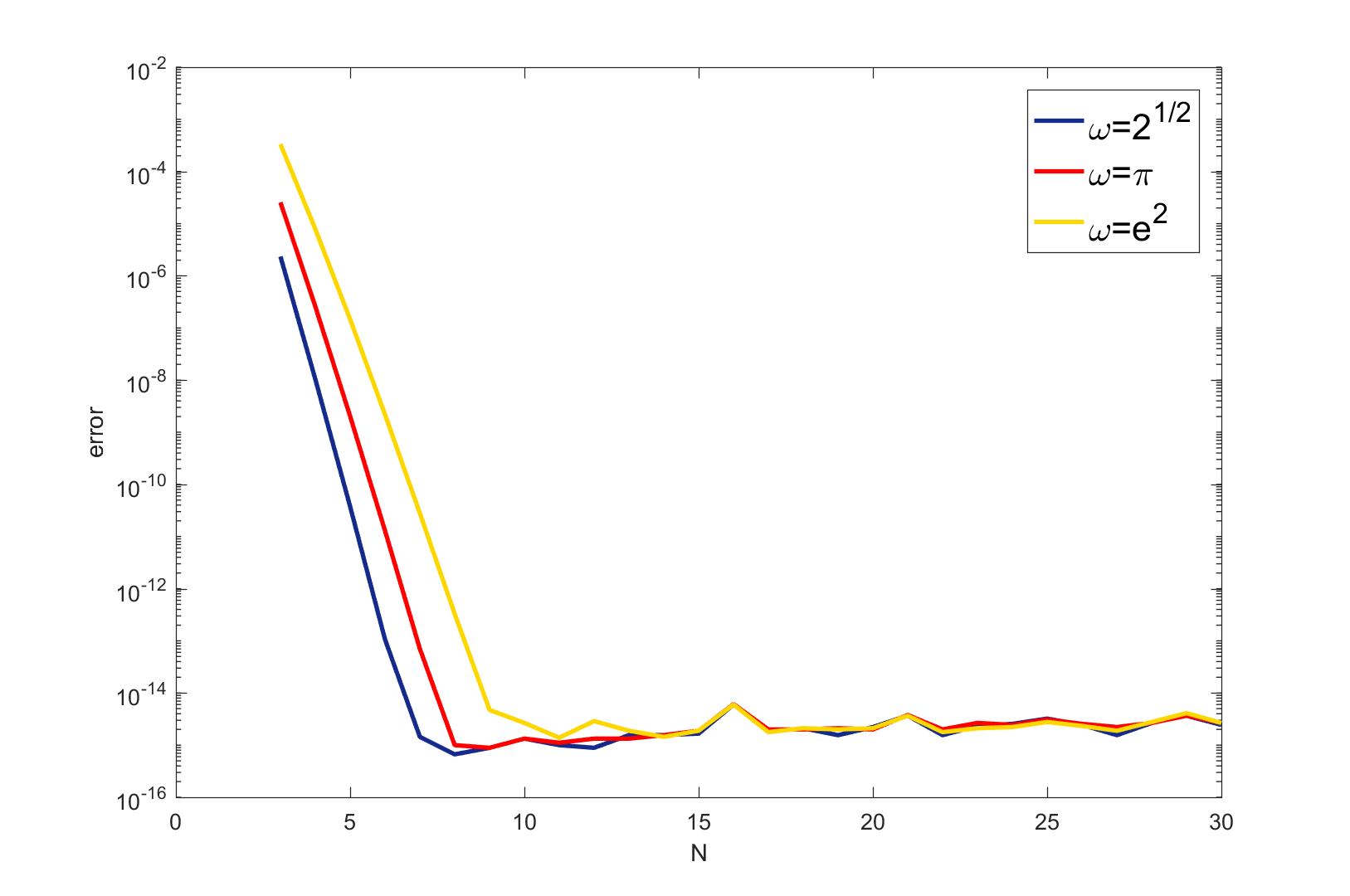}}
}%
\subfigure[$\omega=60$, $K=4$, $\gamma=1$]{
\resizebox*{6cm}{!}{\includegraphics{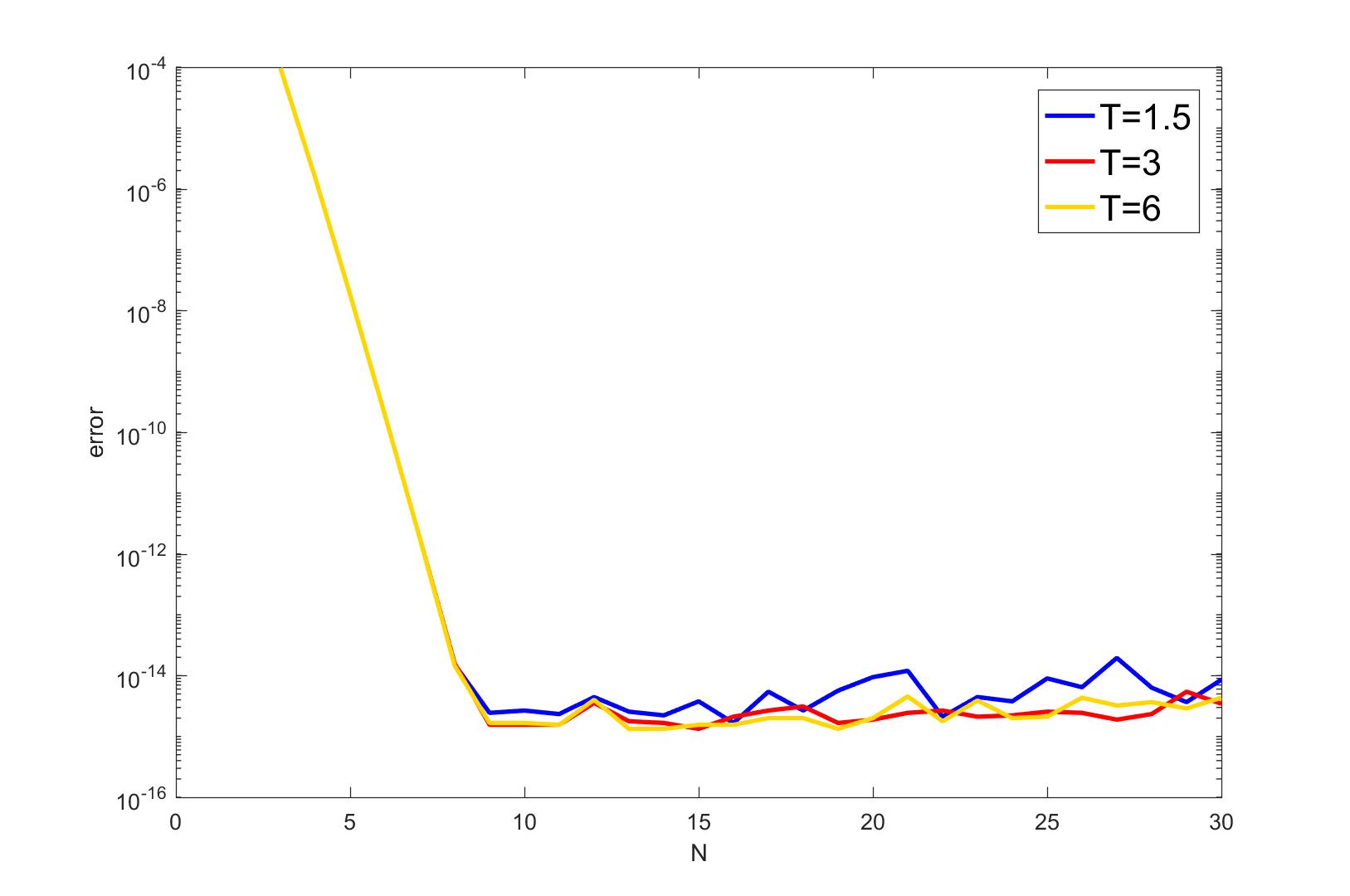}}
}

\subfigure[$\omega=20$, $\gamma=8$, $K=20$]{
\resizebox*{6cm}{!}{\includegraphics{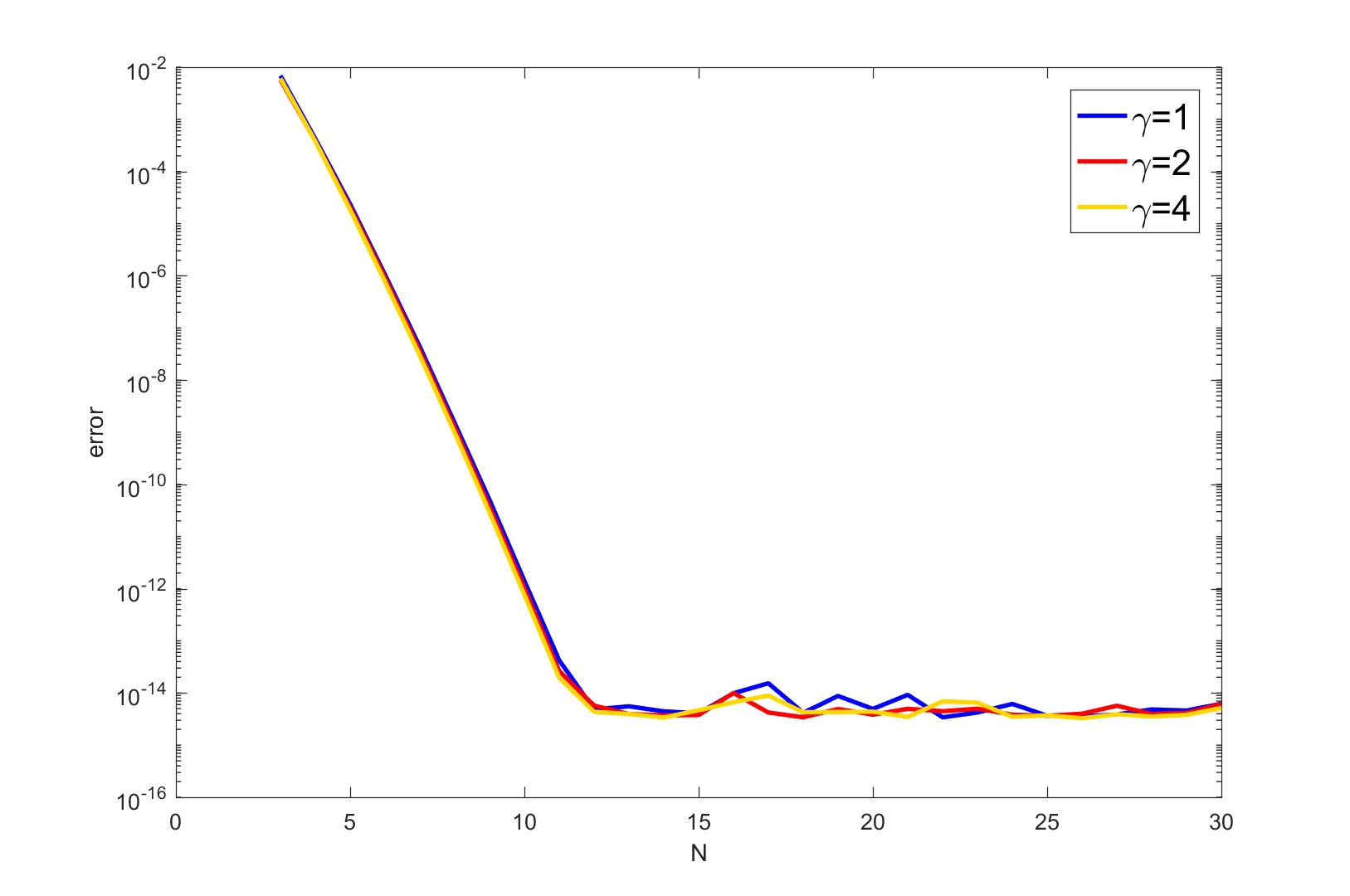}}
}%
\subfigure[$\gamma=4$, $T=8$, $\omega=20$]{
\resizebox*{6cm}{!}{\includegraphics{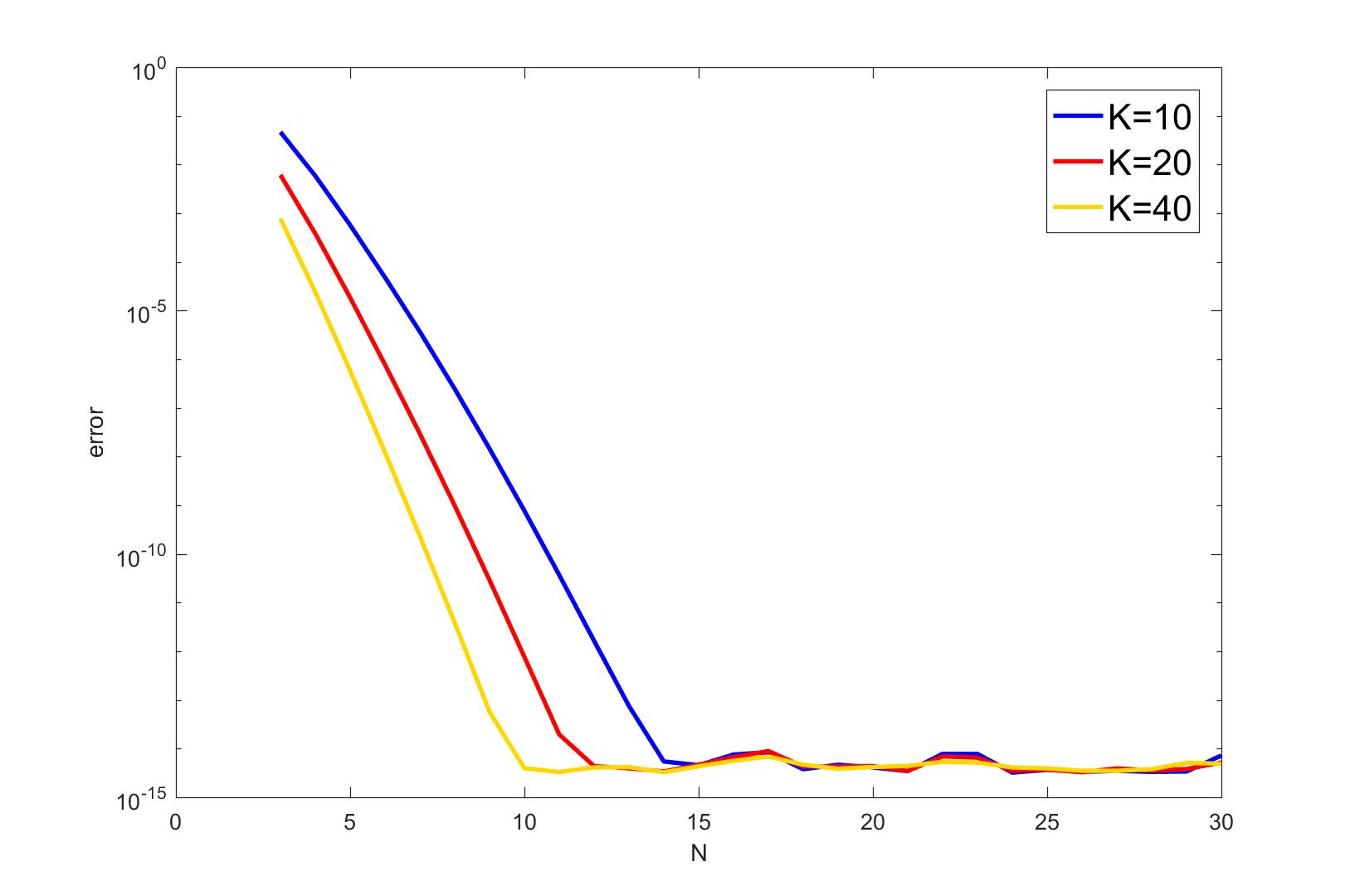}}
}

\caption{Approximation error as a function of the local degree \(N\).
Panel (a) compares different values of the extension parameter \(T\), panel (b) compares different oversampling ratios \(\gamma\), panel (c) compares different numbers of subintervals \(K\), and panel (d) compares different frequencies \(\omega\). The results indicate that the degree required to attain high accuracy is determined mainly by the effective local oscillatory complexity, while its dependence on \(T\) and \(\gamma\) is comparatively weak once these parameters are chosen in a suitable range.}
\label{fig:Ntest}
\end{center}
\end{figure}

\begin{table}[htbp]
\begin{center}
\caption{Numerically observed values of \(m\) and \(C_\Delta\) for different local frequencies \(\omega_\Delta\) when \(\gamma=1\).}
\label{tab:NCw}
\small
{\begin{tabular*}{\textwidth}{@{\extracolsep\fill}cccccccccccccc} \toprule
&$\omega_{\Delta}=0.1$&$\omega_{\Delta}=0.25$&$\omega_{\Delta}=0.5$&$\omega_{\Delta}=1$
&$\omega_{\Delta}=2$&$\omega_{\Delta}=4$&$\omega_{\Delta}=8$&$\omega_{\Delta}=16$\\\midrule
$m$&$9$&$11$&$13$&$15$&$26$&$39$&$65$&$116$\\
$C_{\Delta}$&$9$&$11$&$13$&$13$&$17$&$20$&$26$&$35$\\
$\frac{m C_{\Delta}}{\omega_{\Delta}}$&$810$&$484$&$338$&$195$&$221$&$195$&$211.25$&$253.75$\\
\bottomrule
\end{tabular*}}
\end{center}
\end{table}

\subsection{Computational complexity}

We briefly analyze the computational complexity of the proposed LLF method. A key feature is that, once the local discretization parameters are fixed, all subintervals share the same reference matrix and its truncated singular value decomposition (TSVD). Hence the matrix construction and factorization are performed only once in an offline stage.

The offline cost consists of assembling the matrix \(A_{m,N}^{(T)}\) and computing its SVD, which requires \(O(mN)\) and \(O(mN\min\{m,N\})\) operations, respectively. Since \(m\) and \(N\) are local parameters, this cost is independent of the number of subintervals.

For each subinterval, the online computation involves projection, scaling, and reconstruction in the truncated SVD space. If \(C_\Delta\) denotes the number of retained singular values, the cost per subinterval is
\[
O(C_\Delta m + C_\Delta N),
\]
and the total online cost is therefore
\[
O\bigl(K C_\Delta m\bigr).
\]

Since typically \(C_\Delta \ll N\), the TSVD truncation improves both stability and efficiency. Together with the total number of nodes \(K(m-1)+1\), this explains the relevance of the indicator \(mC_\Delta/\omega_\Delta\) in parameter selection. This linear scaling with respect to the number of subintervals is a key advantage of the proposed local formulation.

\subsection{Practical parameter choice}

We summarize a practical parameter strategy based on the preceding analysis and numerical results. Throughout, we target the accuracy level \(5\times10^{-13}\), which reflects attainable precision in finite arithmetic.

First, the extension parameter \(T\) is chosen mainly for stability. From Table~\ref{tab:T1}, the admissible lower bound \(T_1\) depends on the oversampling ratio \(\gamma\), and for \(\gamma=1\) we have \(T_1\approx 5.6\). We therefore adopt
\[
T=6
\]
as a robust default.

Second, the oversampling ratio \(\gamma\) primarily controls sampling redundancy. Since increasing \(\gamma\) enlarges the number of sampling points without significantly reducing the required degree \(N\), and since the product \(T_1\gamma\) does not improve for larger \(\gamma\), the most efficient choice is
\[
\gamma=1.
\]

Third, the local degree \(N\) is determined by the effective local frequency \(\omega_\Delta\). For low-to-moderate oscillations, the choice
\[
N=15,\qquad m=15
\]
already provides reliable high accuracy. For highly oscillatory functions, a larger degree such as
\[
N=m=40
\]
is preferable, as it reduces the required number of subintervals and hence the total number of nodes.

In summary, we recommend the parameter choice
\[
\gamma=1,\qquad T=6,
\]
together with
\[
N=15 \ \text{for low-to-moderate local frequencies},\qquad
N=40 \ \text{for highly oscillatory cases}.
\]
This simple rule provides an effective guideline for implementing the LLF method in practice.

\section{Numerical Experiments}\label{sec:numerics}

We present numerical experiments to assess the approximation properties of the proposed local Legendre frame (LLF) method. The tests include relatively smooth functions, highly oscillatory functions, comparisons under a fixed sampling budget, and continuous piecewise smooth functions with derivative singularities. Comparisons with the local Fourier extension (LFE) method are included throughout.

\subsection{Experimental setting}

Unless otherwise stated, LLF uses
\[
\gamma=1,\qquad T=6.
\]
For relatively smooth functions we take
\[
m=15,
\]
while for oscillatory functions we test
\[
m=15 \quad\text{and}\quad m=40,
\]
corresponding to \(N=15\) and \(N=40\).

For LFE we use
\[
m=21,\qquad \gamma=1,\qquad T=6.
\]

Let
\[
M=K(m-1)+1
\]
be the total number of sampling points. The error is measured by
\[
E_M=\max_{x\in X_{10M}} |f(x)-f_M(x)|,
\]
where \(X_{10M}\) is a refined grid with approximately \(10M\) points.

\subsection{Relatively smooth functions}
We consider the following smooth test functions:
\begin{align}
f_1(x)&=\frac{\sqrt{x}}{x^2+\frac{x}{2}+1}, \qquad x\in[1,15],\\
f_2(x)&=\frac{1}{4-x^2}, \qquad x\in[-1,1],\\
f_3(x)&=\mathrm{erf}\!\left(\frac{x}{3}\right), \qquad x\in[-4,4].
\end{align}
\begin{figure}[htbp]
\begin{center}
\subfigure[$f_1(x)$]{
\resizebox*{6cm}{!}{\includegraphics{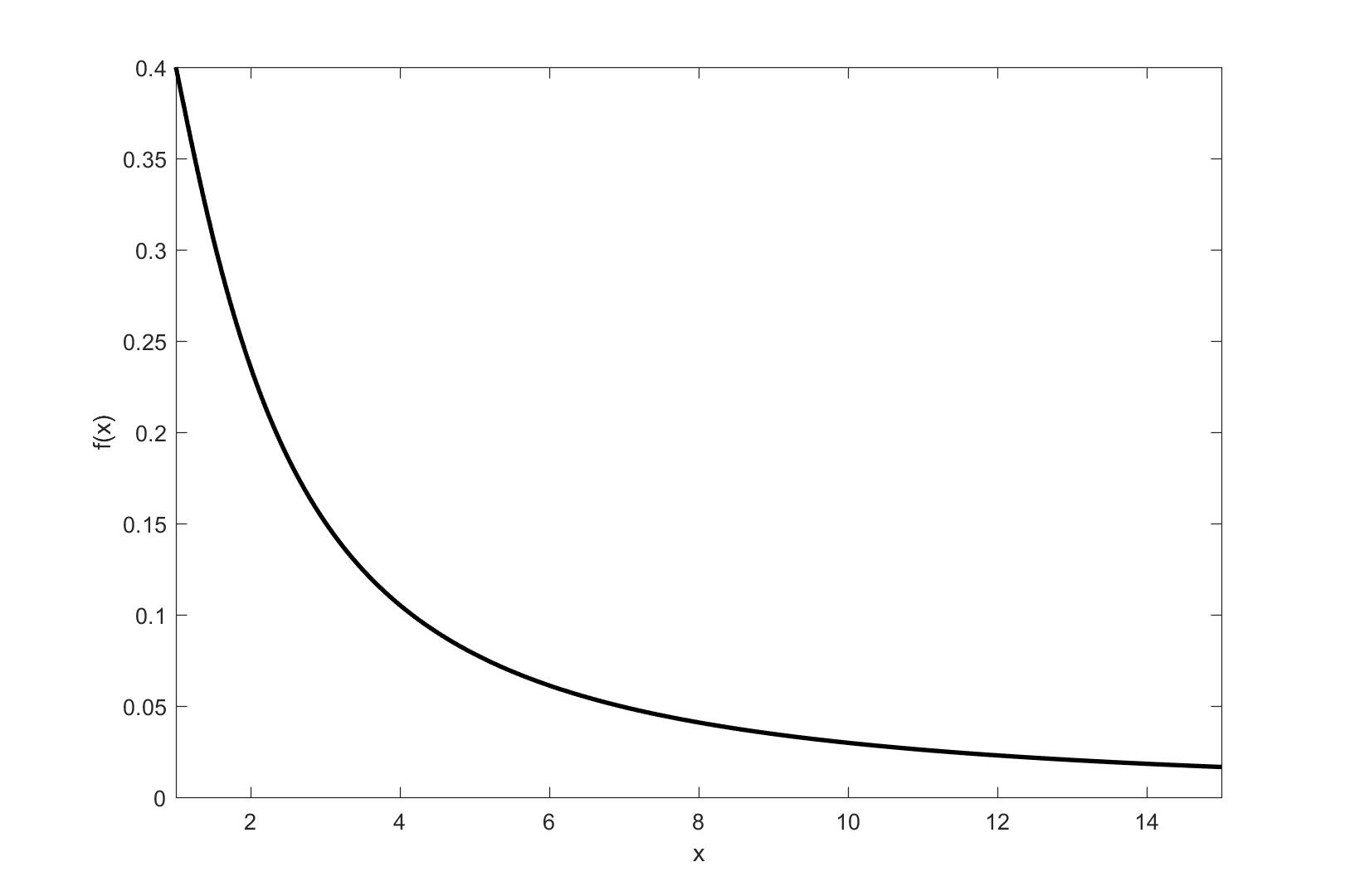}}
}%
\subfigure[Approximation error for \(f_1\)]{
\resizebox*{6cm}{!}{\includegraphics{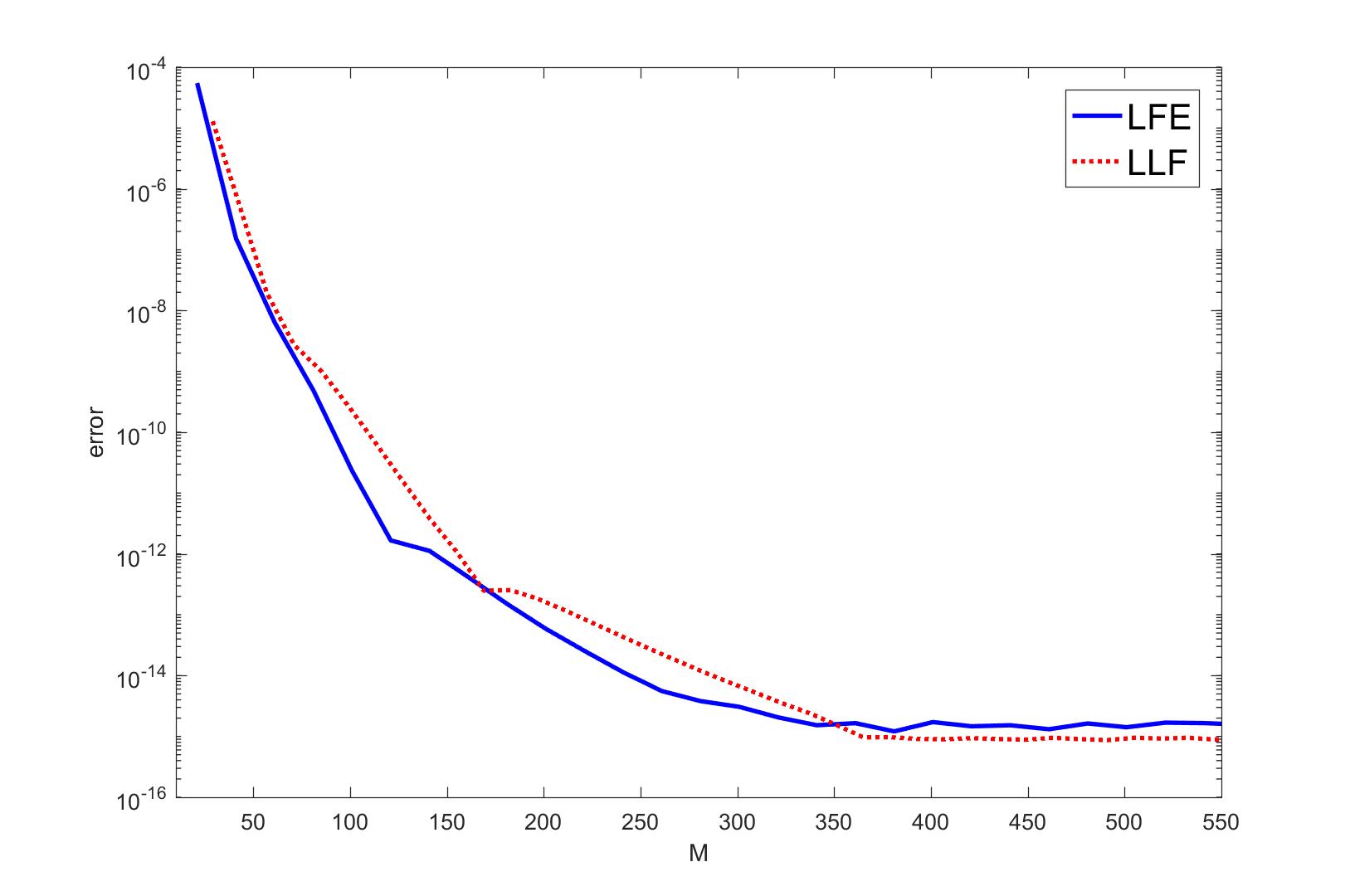}}
}

\subfigure[$f_2(x)$]{
\resizebox*{6cm}{!}{\includegraphics{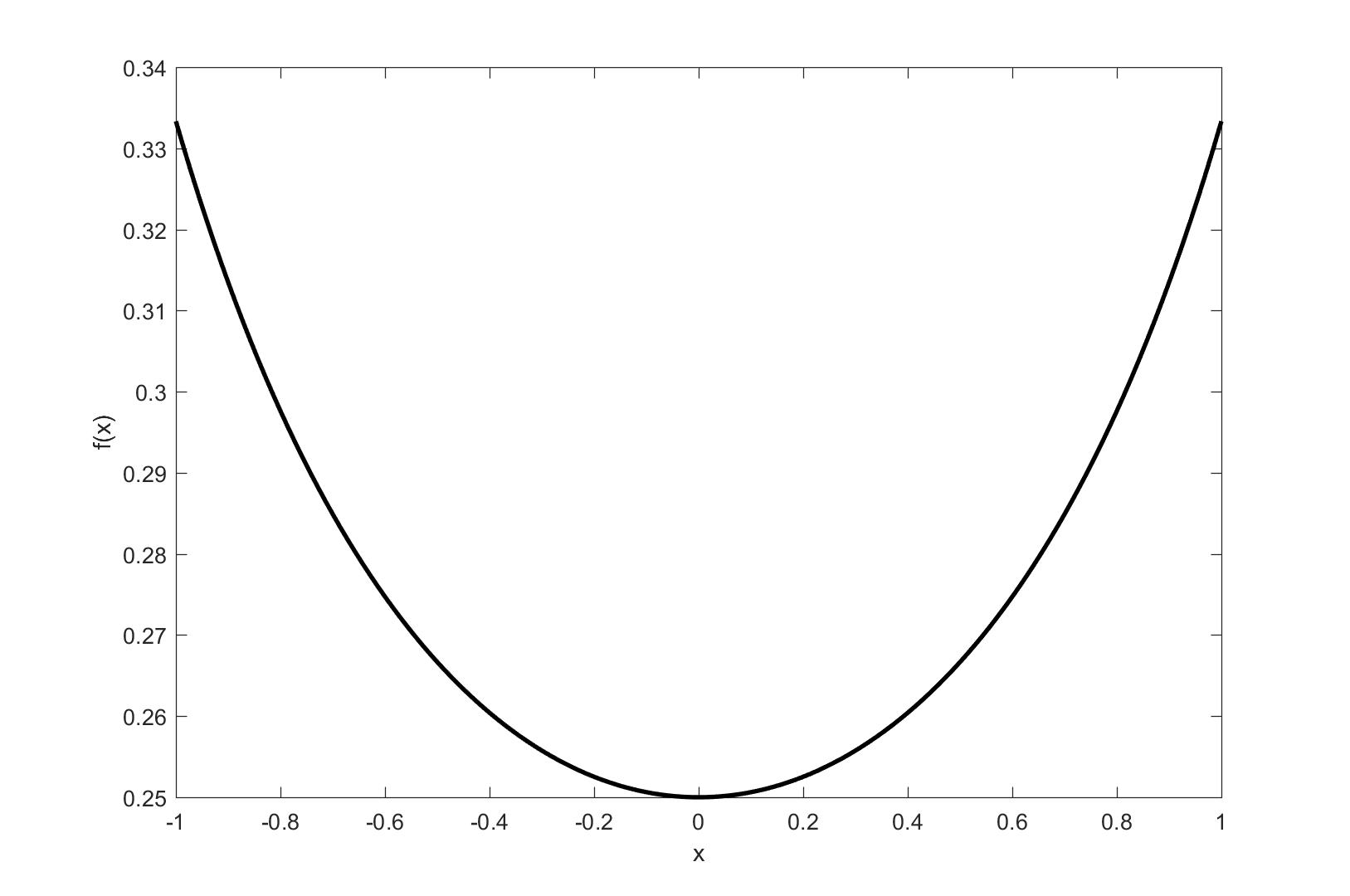}}
}%
\subfigure[Approximation error for \(f_2\)]{
\resizebox*{6cm}{!}{\includegraphics{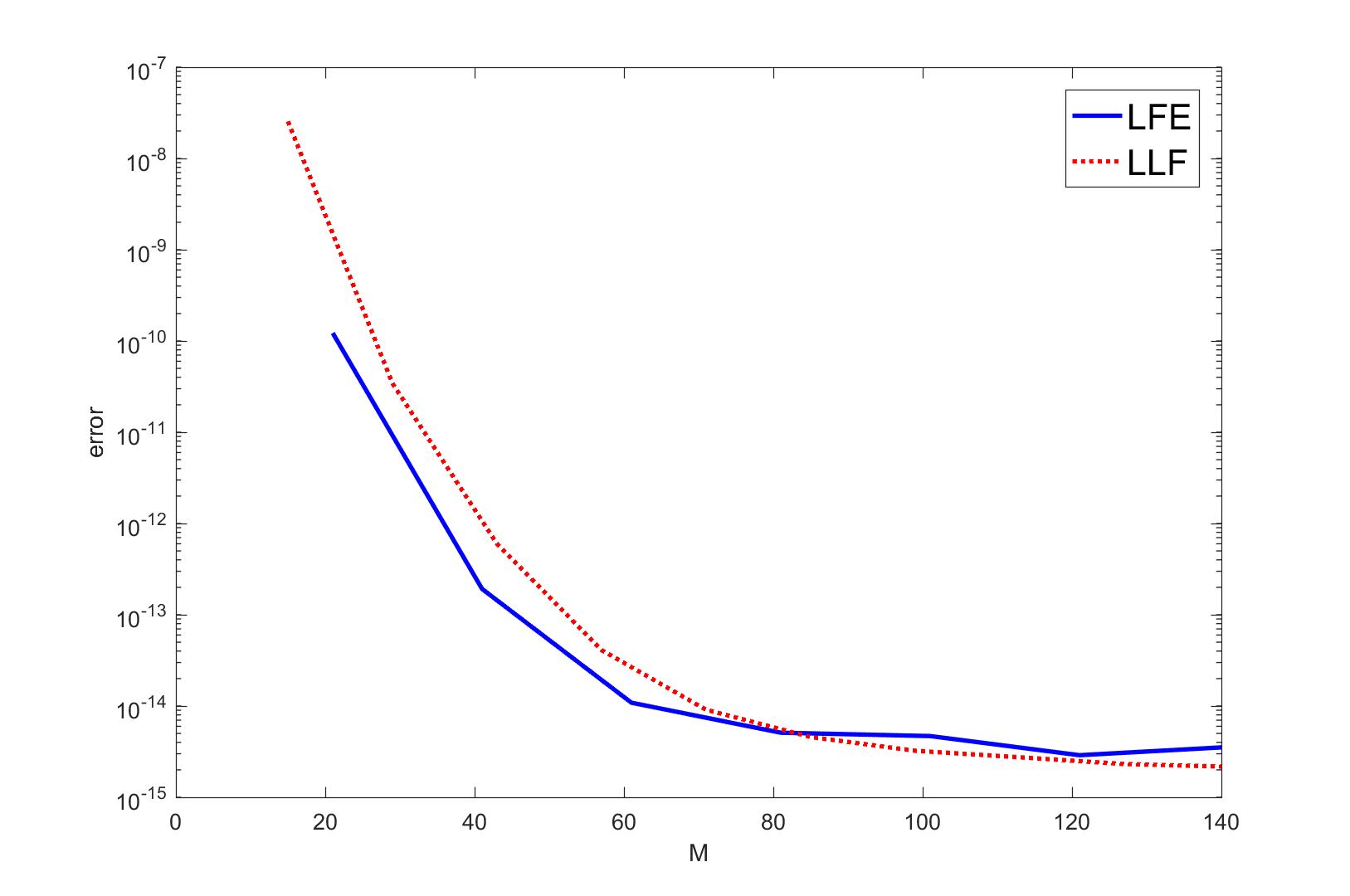}}
}

\subfigure[$f_3(x)$]{
\resizebox*{6cm}{!}{\includegraphics{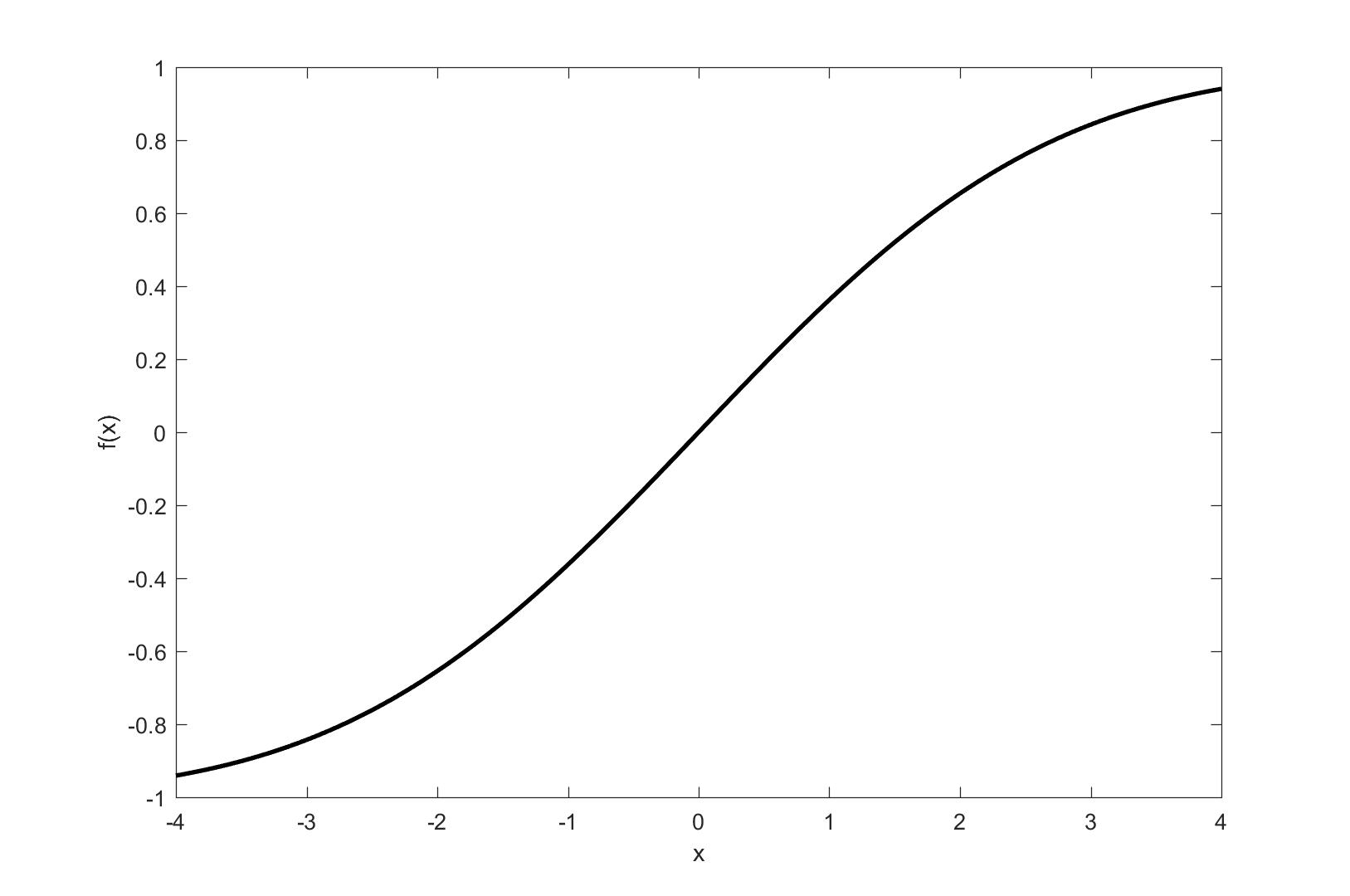}}
}%
\subfigure[Approximation error for \(f_3\)]{
\resizebox*{6cm}{!}{\includegraphics{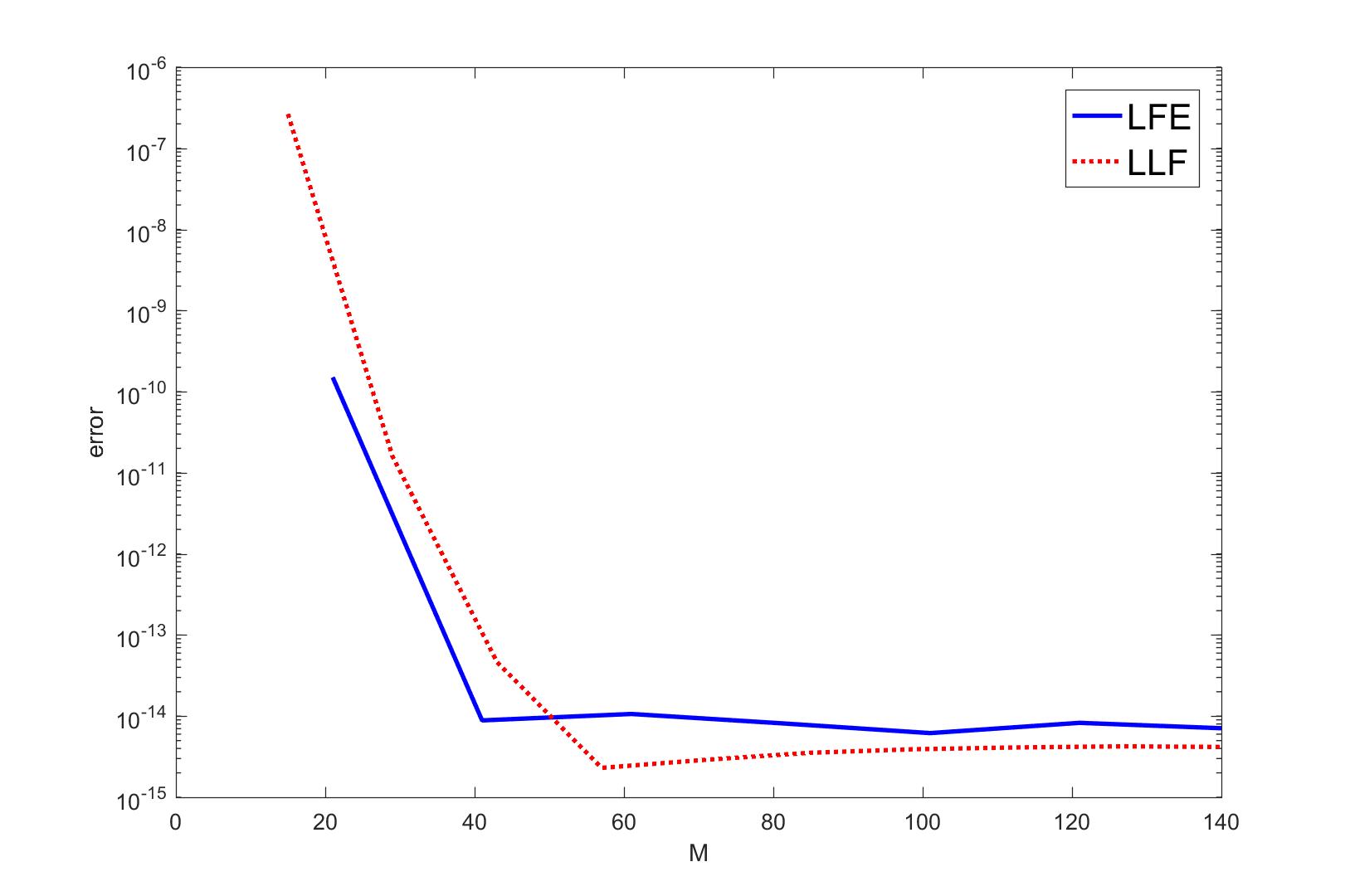}}
}

\caption{Tests on relatively smooth functions. Panels (a), (c), and (e) show the graphs of the test functions \(f_1\), \(f_2\), and \(f_3\), respectively; panels (b), (d), and (f) show the corresponding approximation errors versus the total number of sampling points \(M\). In all three cases, LLF and LFE achieve comparable accuracy with similar node counts, while LLF is slightly more economical in online computation because of its lower retained numerical rank.}
\label{fig:smooth_tests}
\end{center}
\end{figure}
The results (Figure~\ref{fig:smooth_tests}) show that both LLF and LFE exhibit steady convergence and reach near-machine accuracy with comparable node counts.

In this regime, the dominant online cost scales like \(C_\Delta M\). Since LLF typically yields a smaller local numerical rank \(C_\Delta\), it is slightly more efficient. Thus, for smooth or slowly varying functions, LLF provides a competitive and economical alternative to LFE.

\subsection{Highly oscillatory functions}

We next consider
\begin{align}
f_4(x)&=\cos\!\left(\frac{100}{1+25x^2}\right), \qquad x\in[-1,1],\\
f_5(x)&=\cos(200x^2), \qquad x\in[-1,1],\\
f_6(x)&=\mathrm{Ai}(76x), \qquad x\in[-1,1].
\end{align}

The results (Figure~\ref{fig:osc_tests}) show a clear difference between LLF and LFE. For all examples, LFE achieves high accuracy with significantly fewer nodes. LLF requires a much finer discretization, even when using \(m=40\).
\begin{figure}[htbp]
\begin{center}
\subfigure[$f_4(x)$]{
\resizebox*{6cm}{!}{\includegraphics{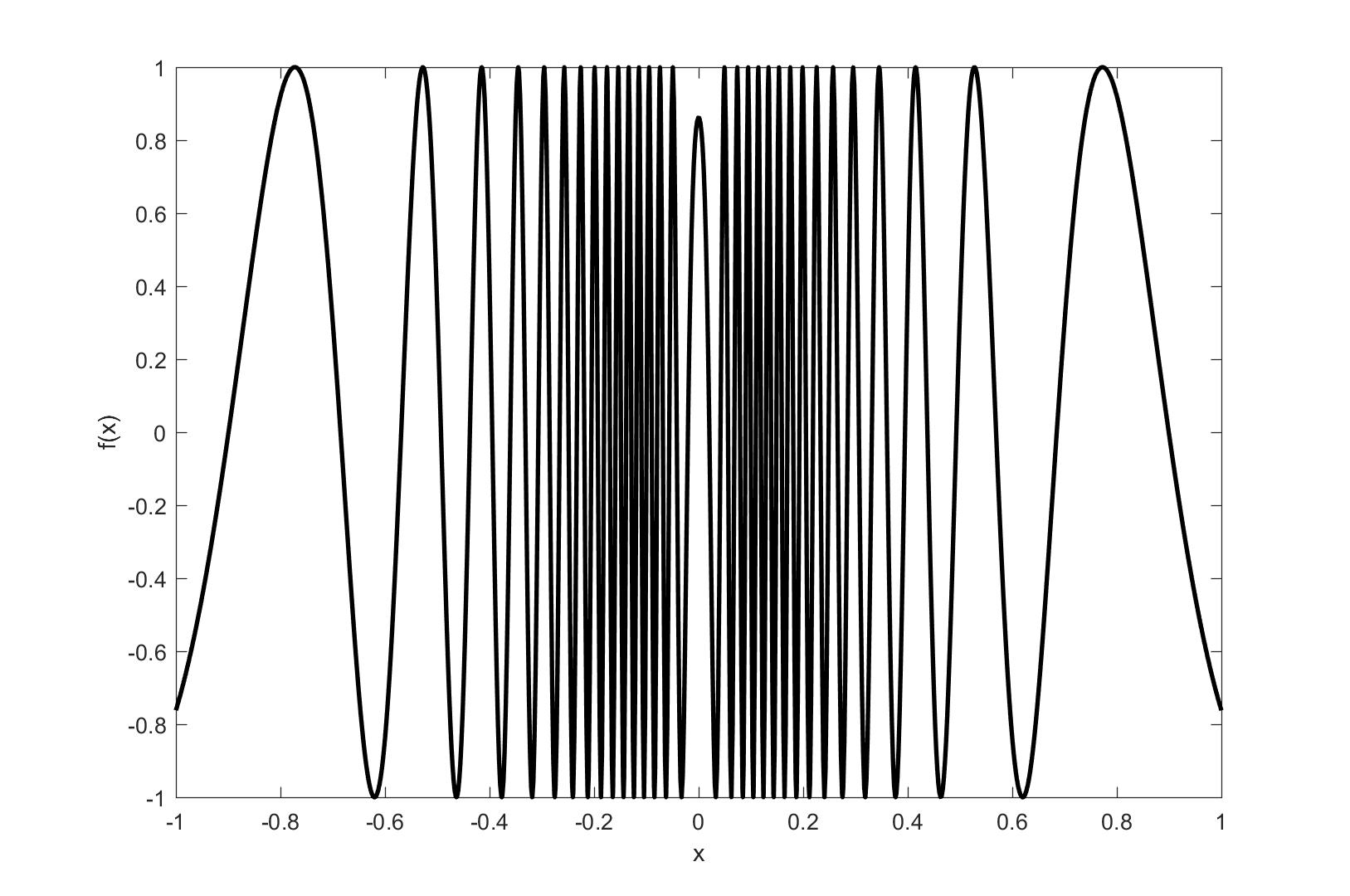}}
}%
\subfigure[Approximation error for \(f_4\)]{
\resizebox*{6cm}{!}{\includegraphics{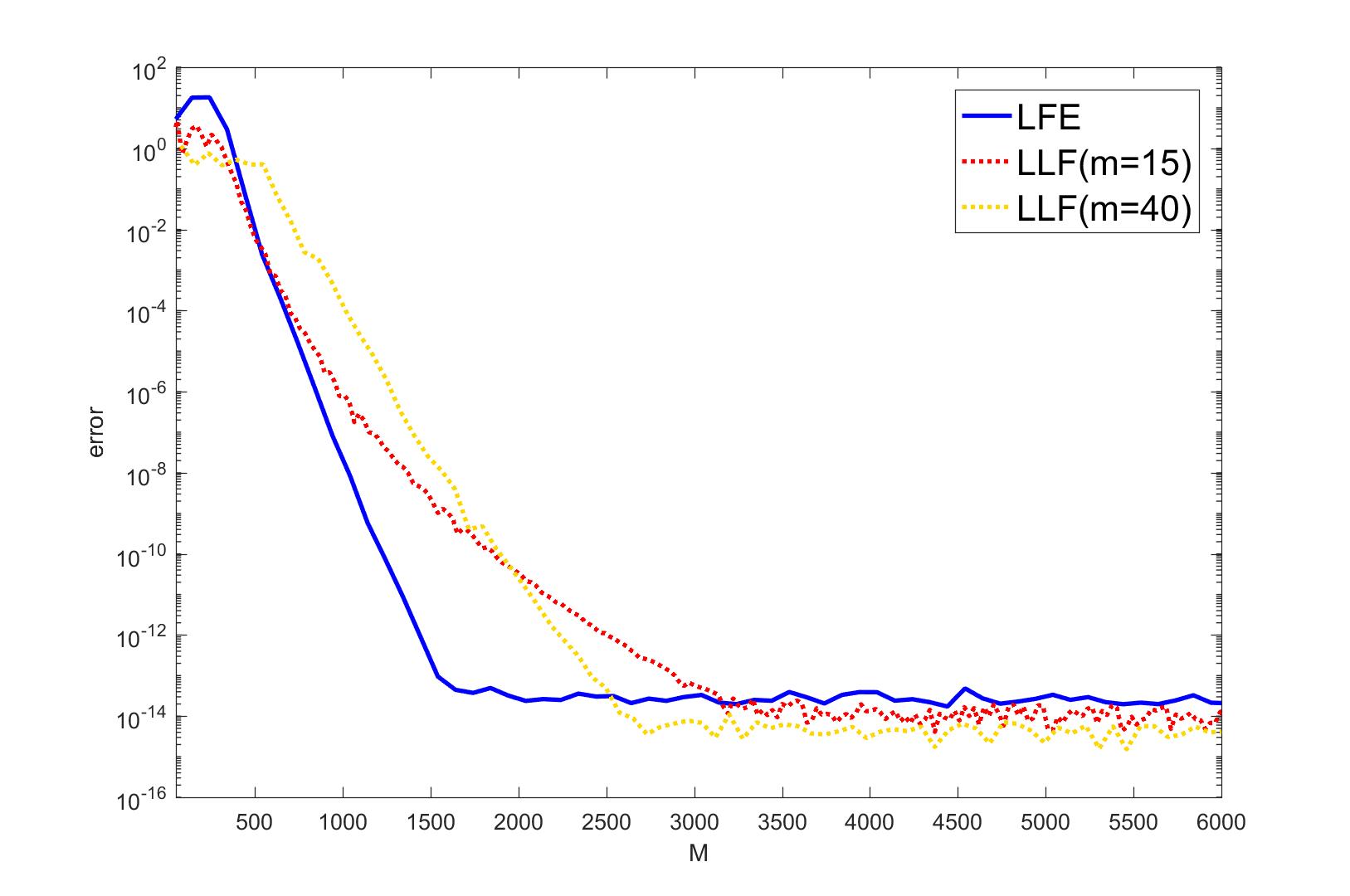}}
}

\subfigure[$f_5(x)$]{
\resizebox*{6cm}{!}{\includegraphics{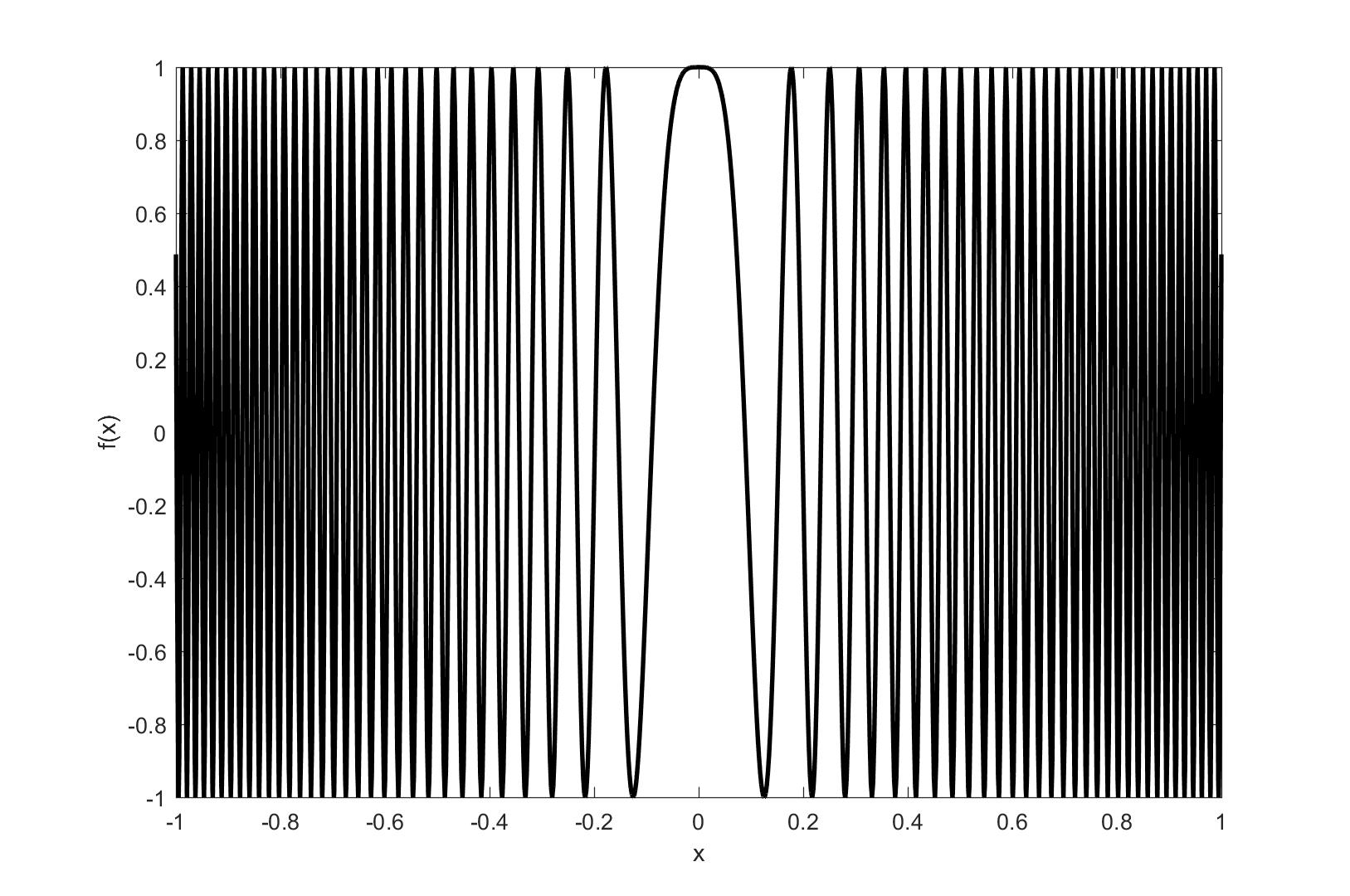}}
}%
\subfigure[Approximation error for \(f_5\)]{
\resizebox*{6cm}{!}{\includegraphics{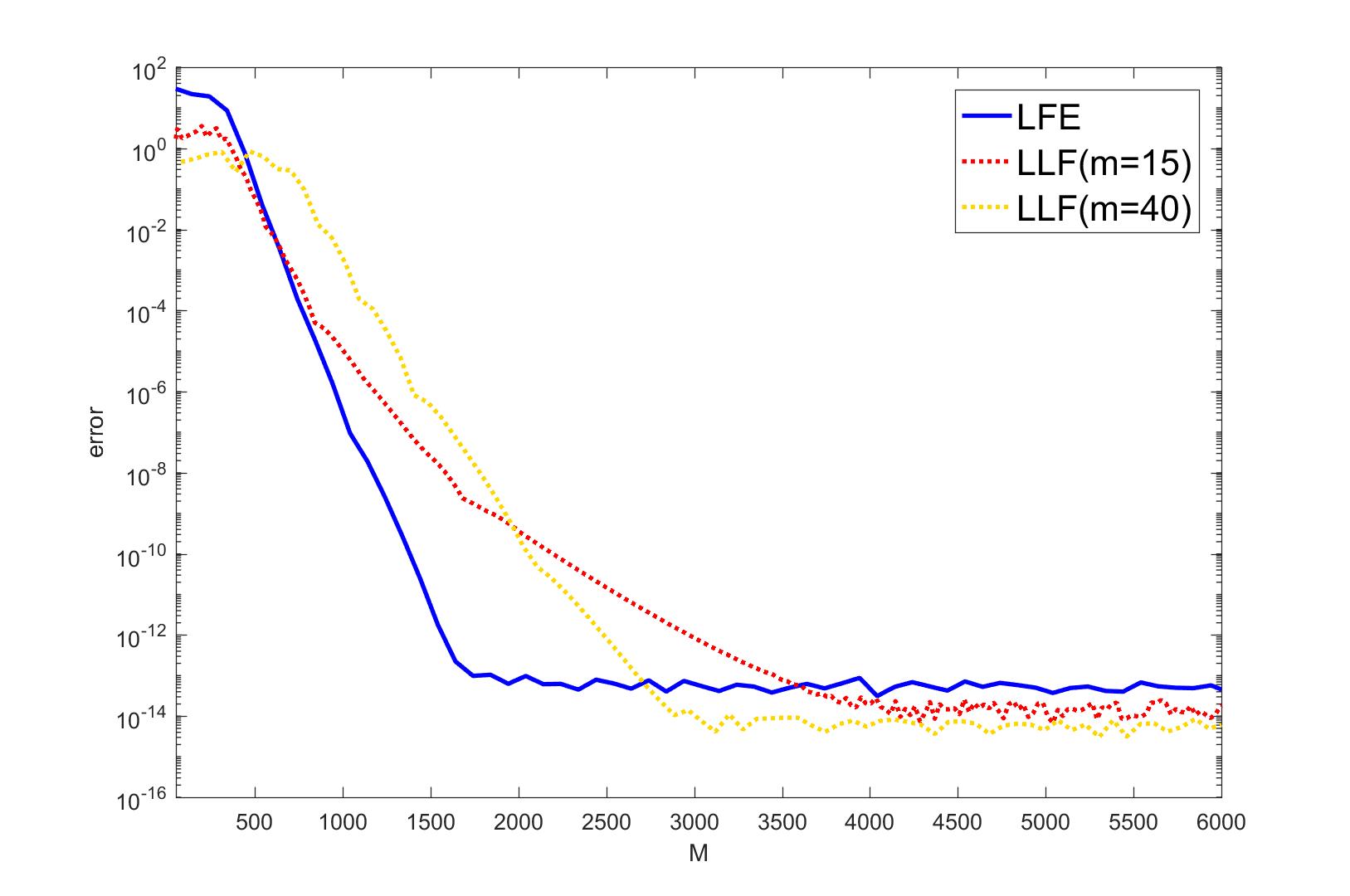}}
}

\subfigure[$f_6(x)$]{
\resizebox*{6cm}{!}{\includegraphics{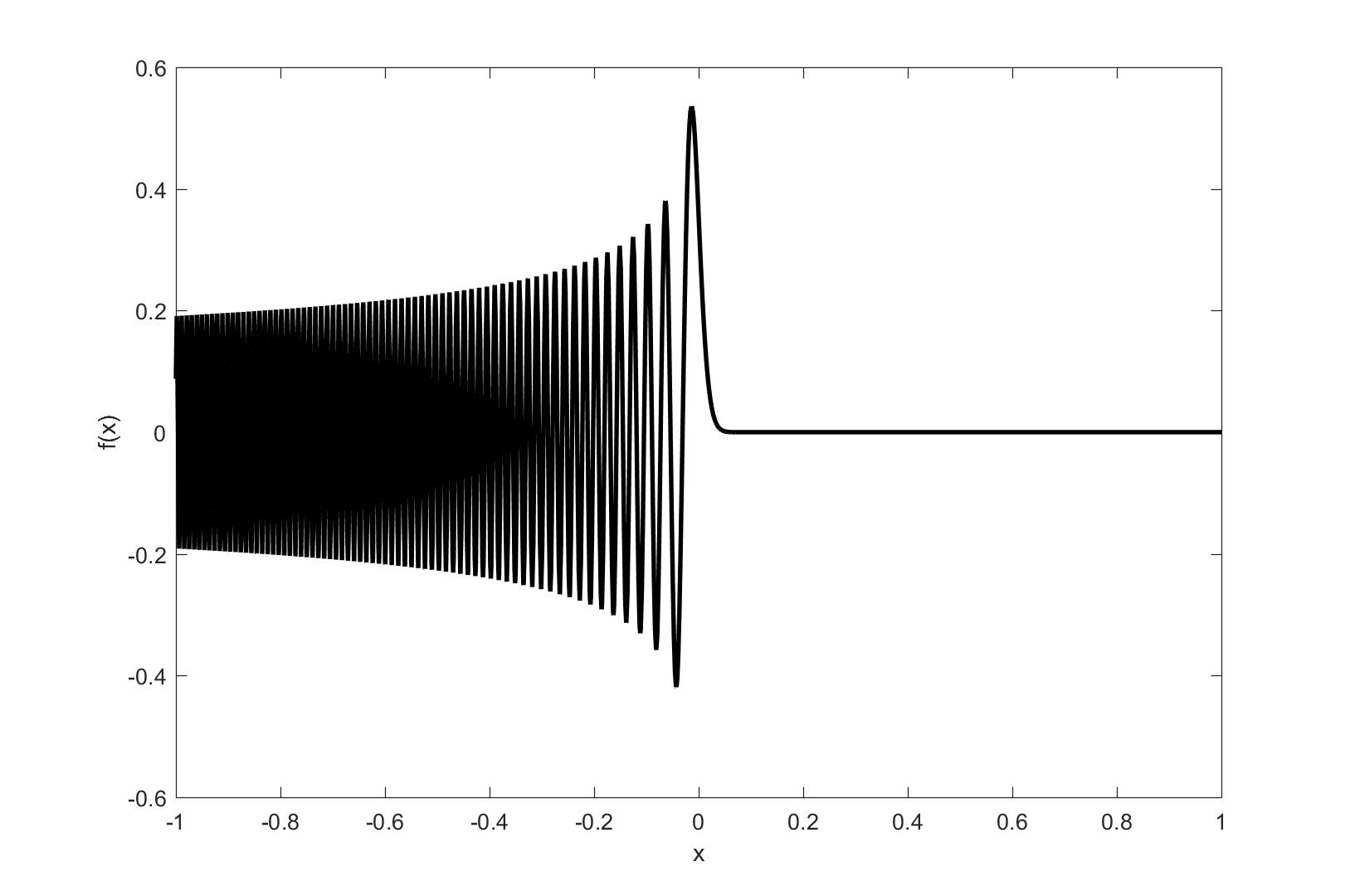}}
}%
\subfigure[Approximation error for \(f_6\)]{
\resizebox*{6cm}{!}{\includegraphics{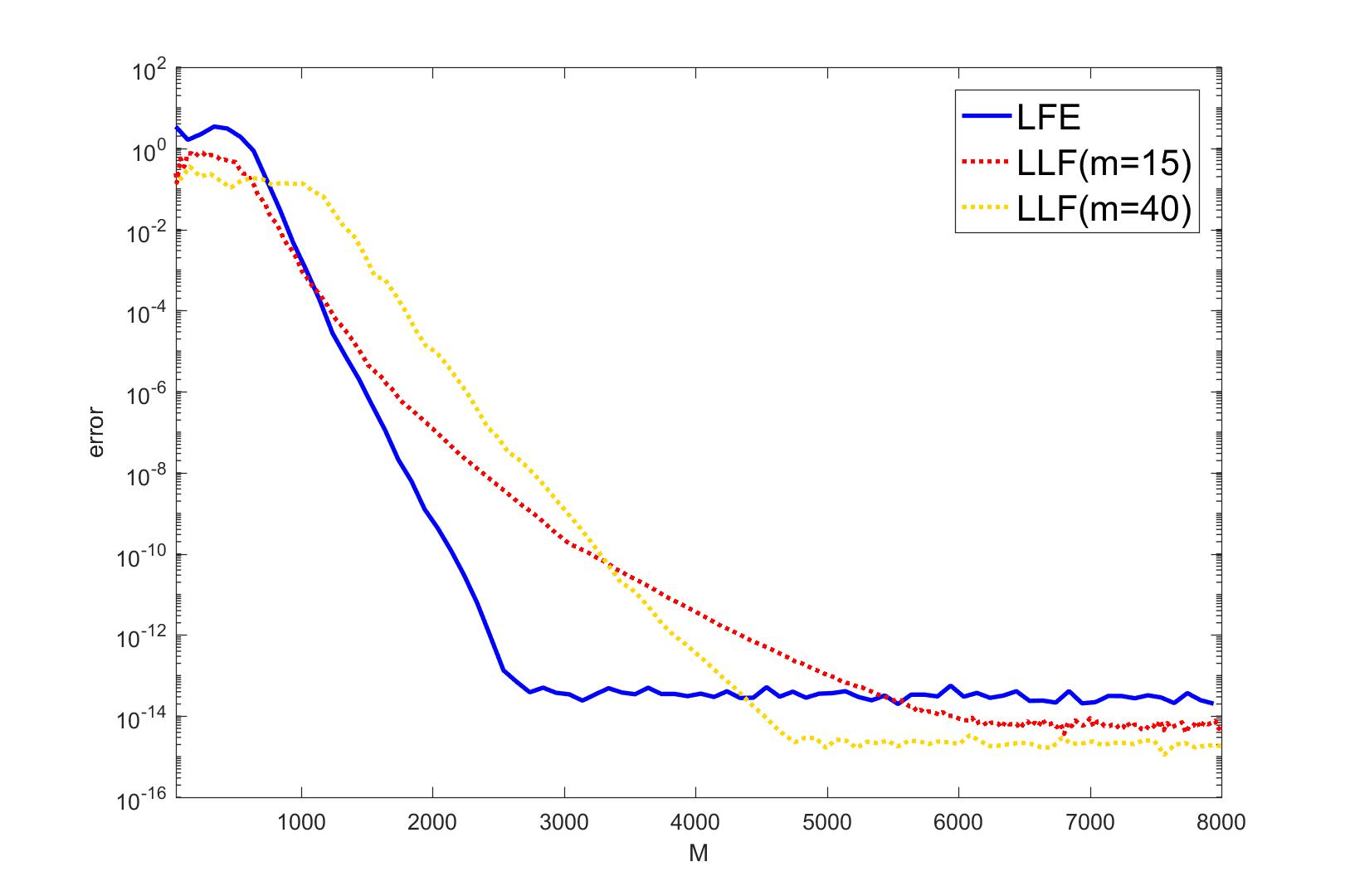}}
}

\caption{Tests on highly oscillatory functions. Panels (a), (c), and (e) show the graphs of the oscillatory test functions \(f_4\), \(f_5\), and \(f_6\), respectively; panels (b), (d), and (f) show the corresponding approximation errors versus the total number of sampling points \(M\). The LLF results are reported for two local choices, \(m=15\) and \(m=40\), and are compared with LFE. The results show that, in the high-frequency regime, LLF requires substantially more total nodes than LFE to attain comparable accuracy.}
\label{fig:osc_tests}
\end{center}
\end{figure}
This reflects the different resolution mechanisms: Fourier-based representations match oscillatory structures directly, whereas polynomial frames require higher local degree or finer subdivision. Consequently, LLF is less efficient in the high-frequency regime.

\subsection{Comparison under a fixed sampling budget}

We further compare LLF and LFE under the same total number of nodes using
\begin{align}
f_7(x)&=\frac{1}{1+25x^2}, \qquad x\in[-1,1],\\
f_8(x)&=\exp\!\bigl(\sin(2.7\pi x)+\cos(\pi x)\bigr), \qquad x\in[-1,1],\\
f_9(x)&=x^2\sin(20x), \qquad x\in[-1,1].
\end{align}

We take \(K=14\) for LFE and \(K=20\) for LLF, so that both use
\[
M=281
\]
nodes.

The results (Figure~\ref{fig:sameM_tests}) show that the two methods achieve nearly identical accuracy. Since LLF uses fewer local nodes per subinterval, its online cost is slightly smaller, whereas LFE tolerates a wider range of oscillatory behavior.

\begin{figure}[htbp]
\begin{center}
\subfigure[$f_7(x)$]{
\resizebox*{6cm}{!}{\includegraphics{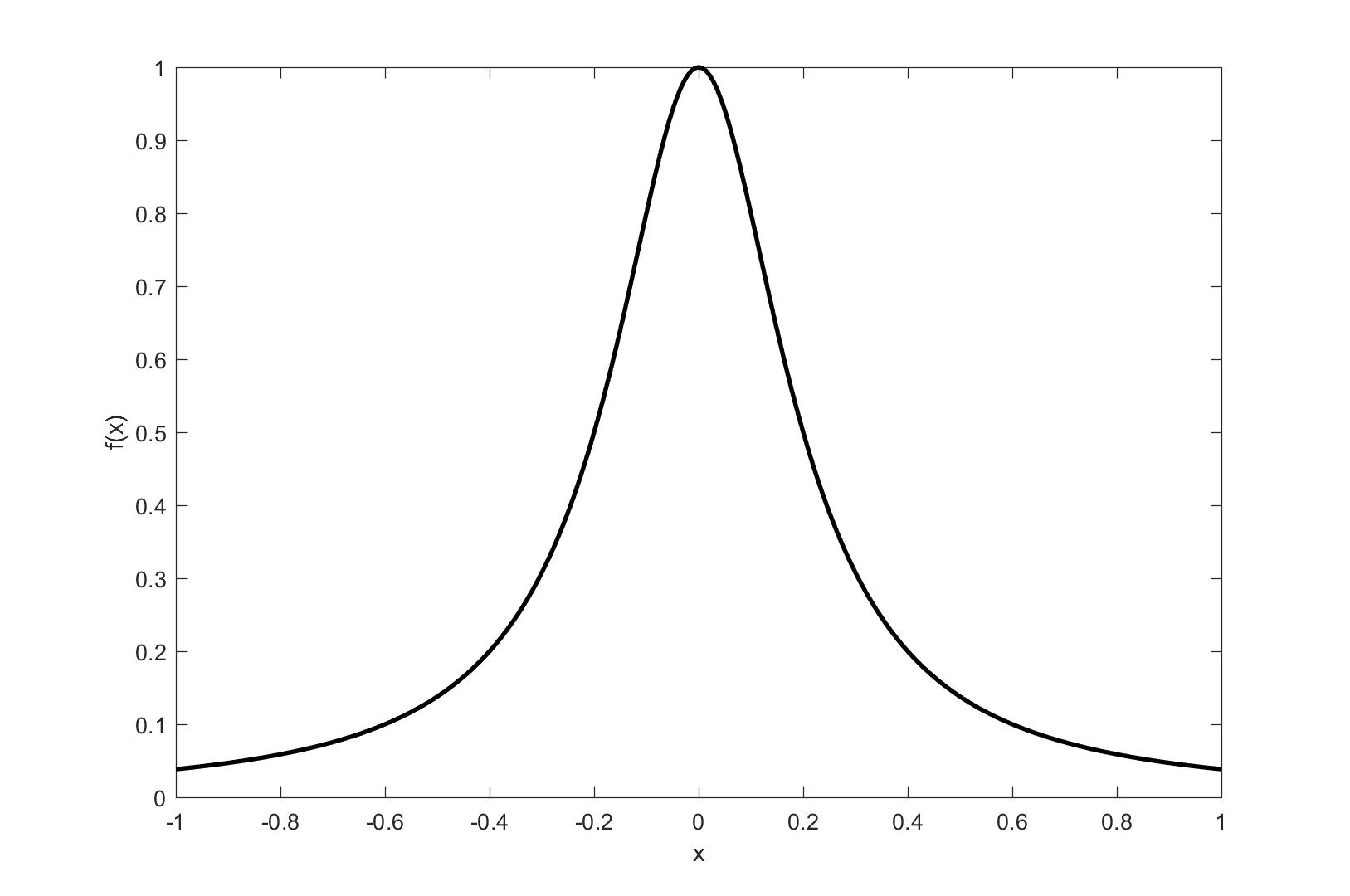}}
}%
\subfigure[Pointwise error for \(f_7\)]{
\resizebox*{6cm}{!}{\includegraphics{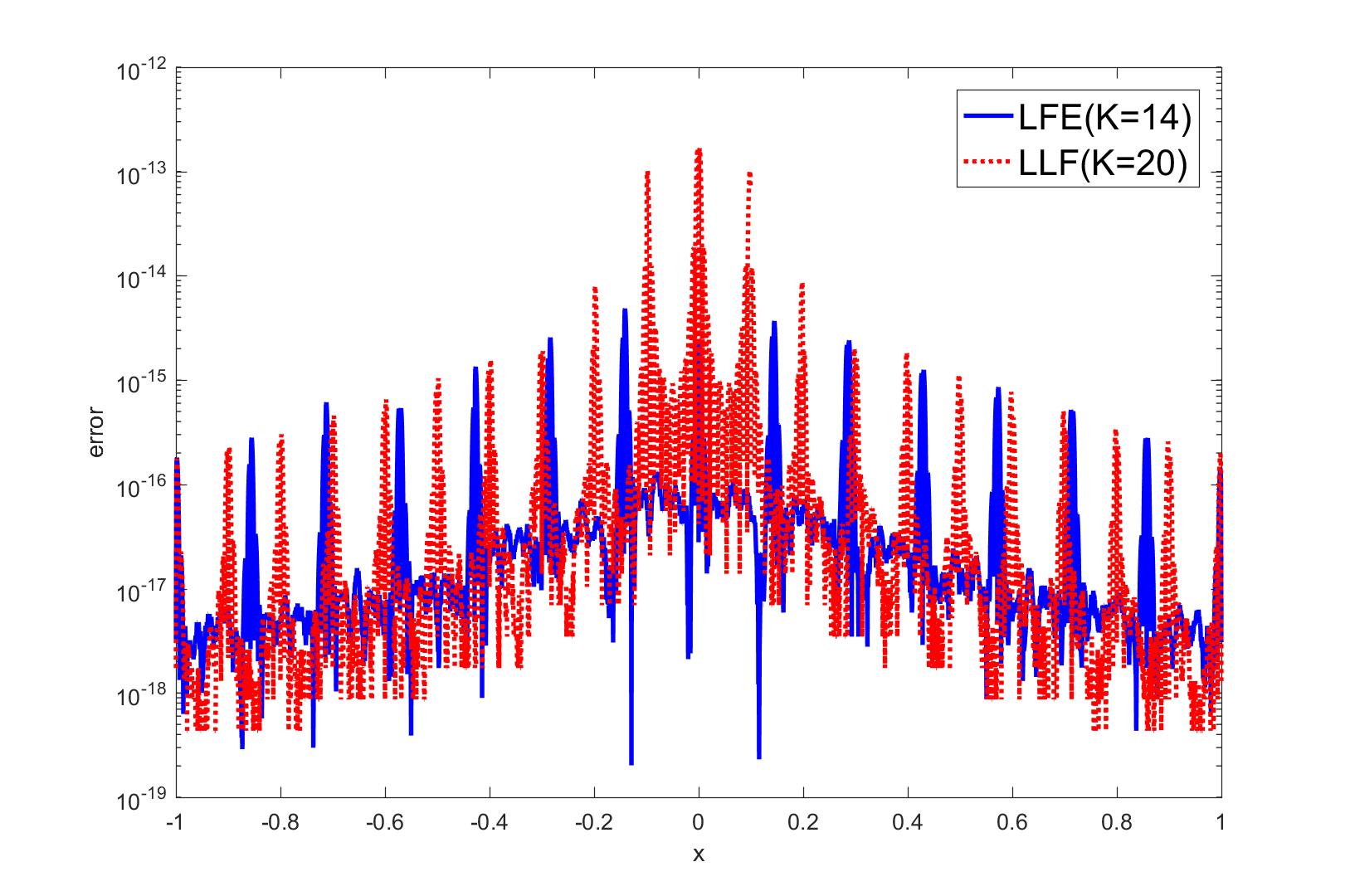}}
}

\subfigure[$f_8(x)$]{
\resizebox*{6cm}{!}{\includegraphics{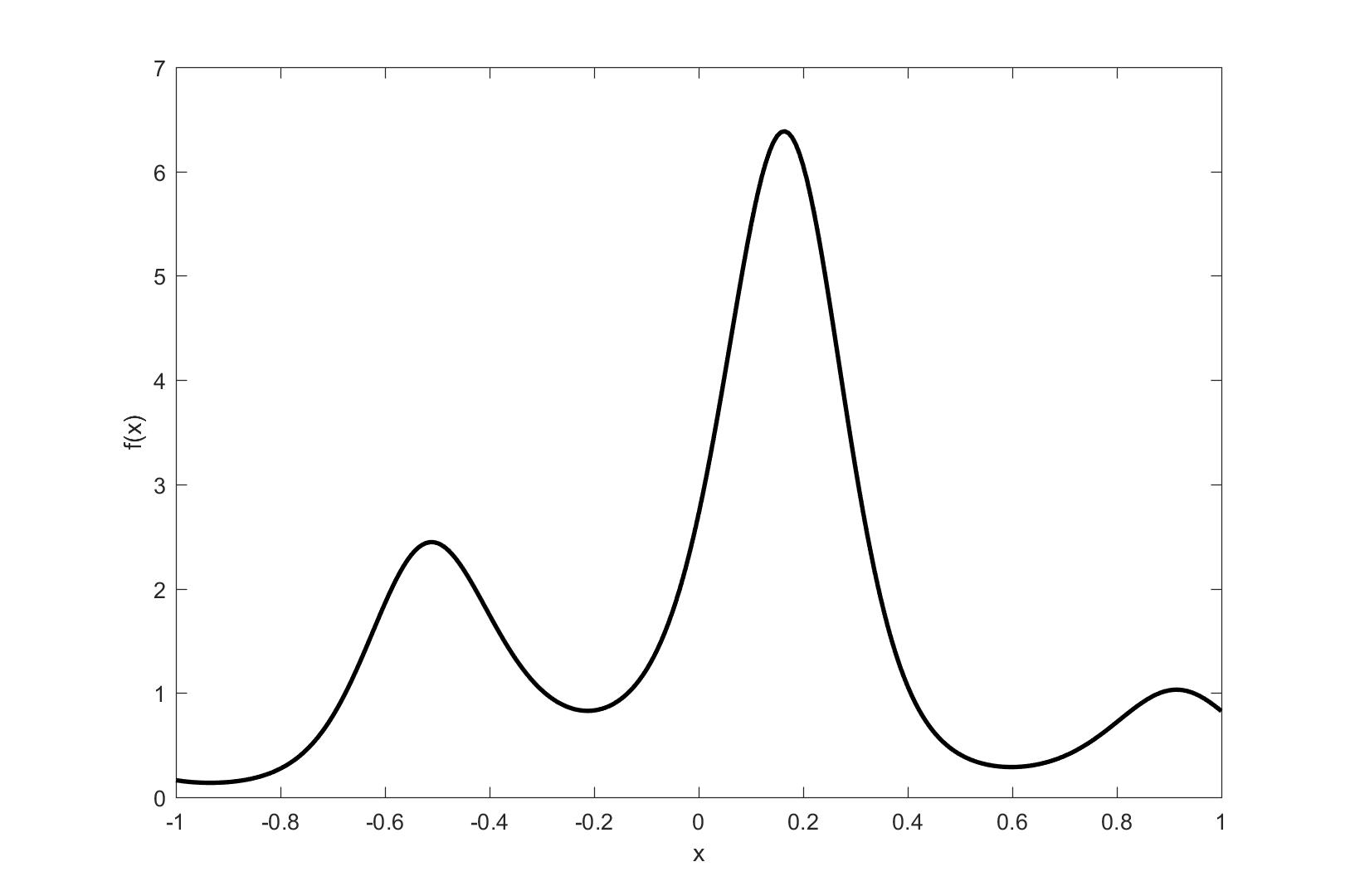}}
}%
\subfigure[Pointwise error for \(f_8\)]{
\resizebox*{6cm}{!}{\includegraphics{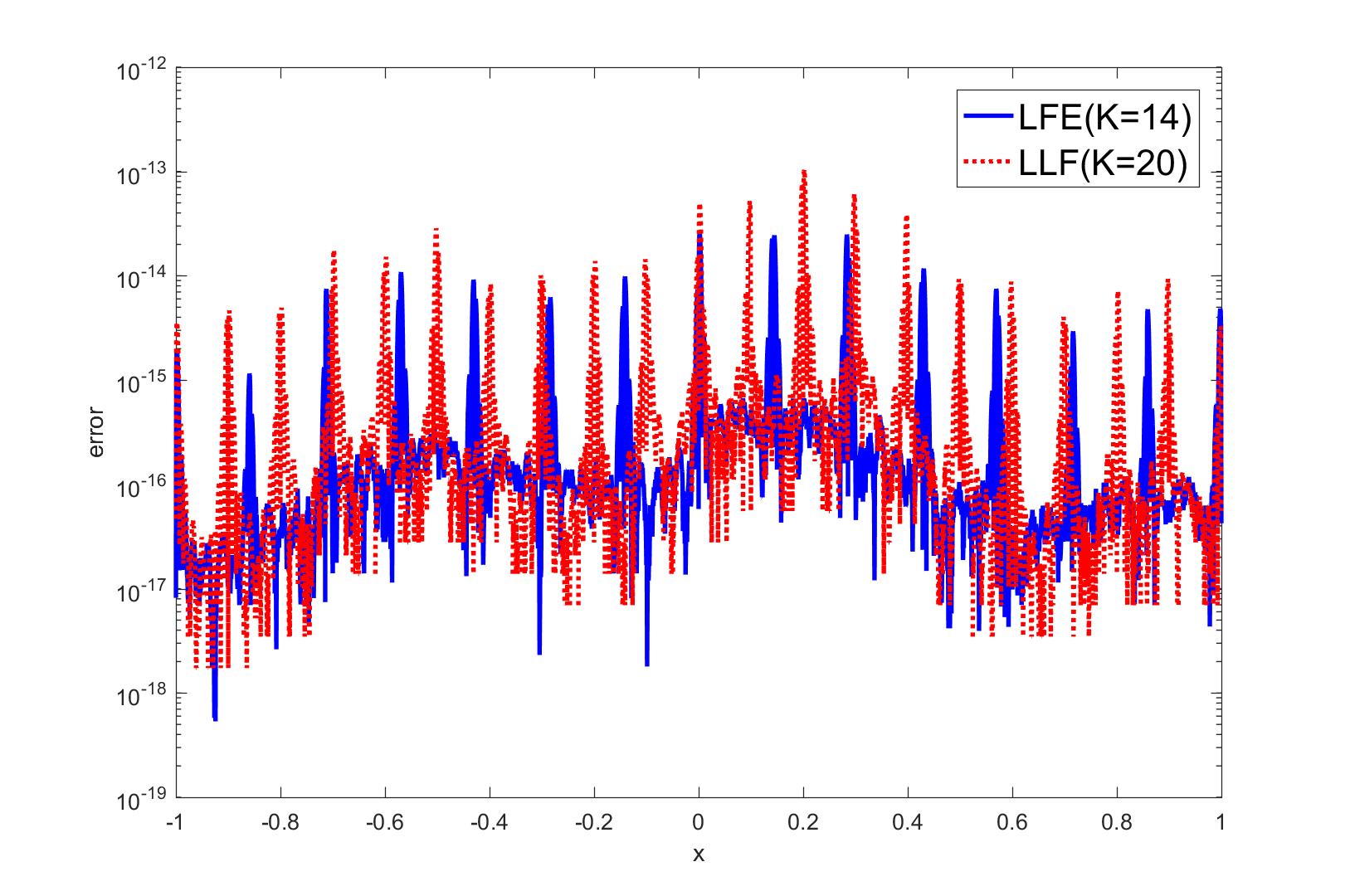}}
}

\subfigure[$f_9(x)$]{
\resizebox*{6cm}{!}{\includegraphics{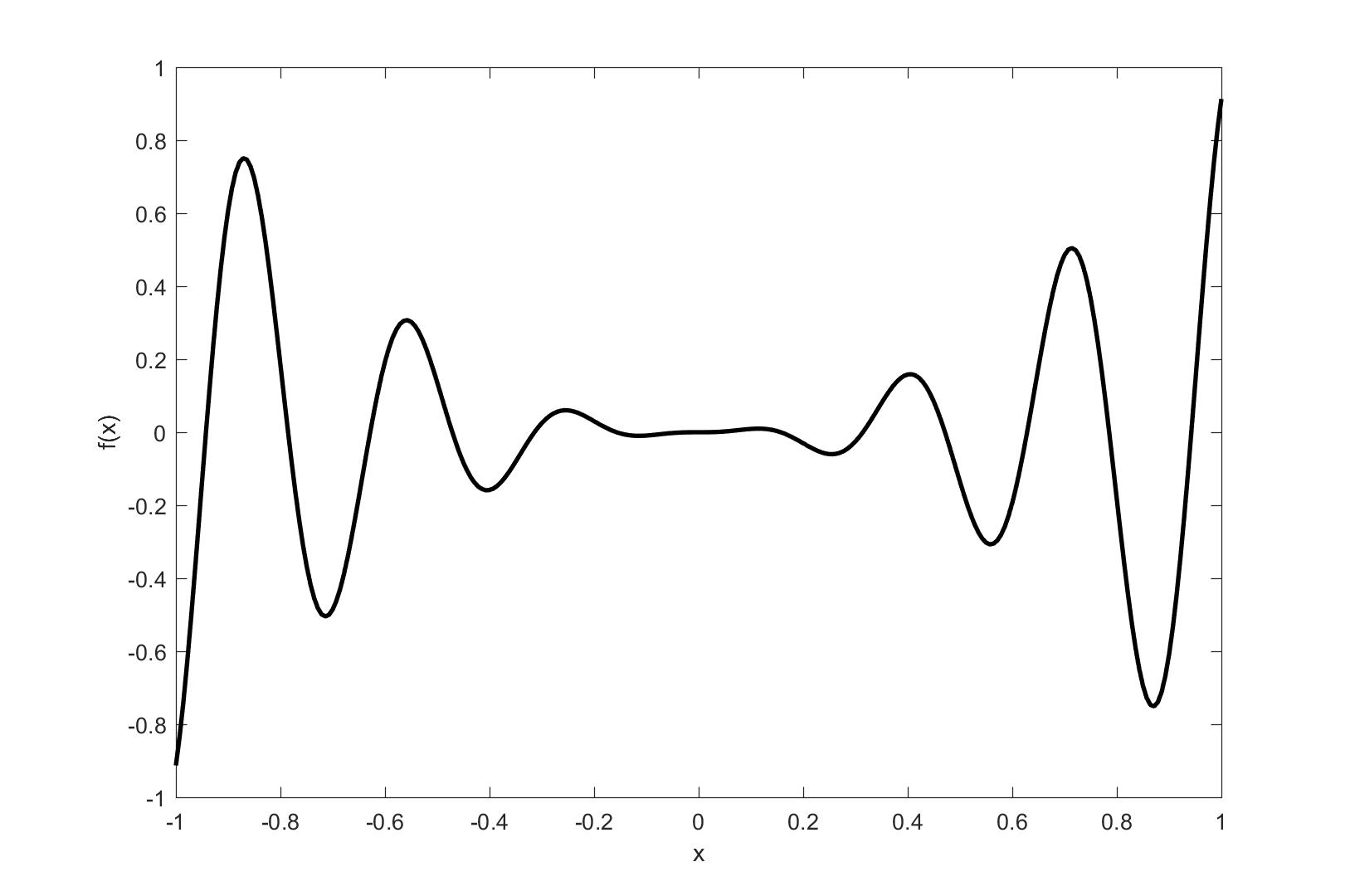}}
}%
\subfigure[Pointwise error for \(f_9\)]{
\resizebox*{6cm}{!}{\includegraphics{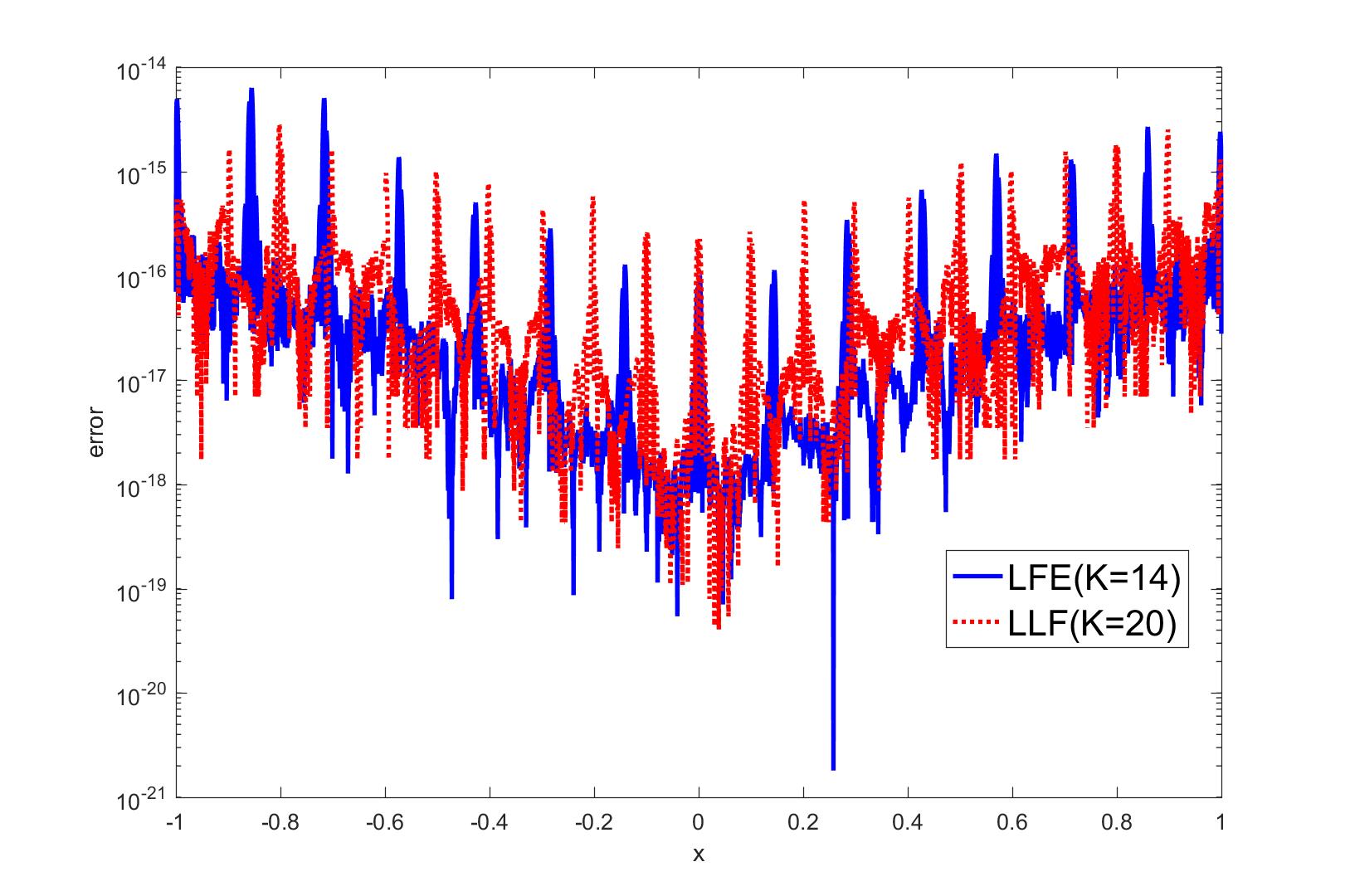}}
}

\caption{Comparison between LLF and LFE under the same total number of sampling points \(M=281\). Panels (a), (c), and (e) show the test functions \(f_7\), \(f_8\), and \(f_9\), respectively; panels (b), (d), and (f) show the corresponding pointwise error distributions. Here LFE uses \(K=14\) with \(m=21\), while LLF uses \(K=20\) with \(m=15\), so that both methods have the same total number of nodes. The results show that, for these moderately oscillatory examples, LLF and LFE attain very similar approximation accuracy under the same sampling budget.}
\label{fig:sameM_tests}
\end{center}
\end{figure}

\subsection{Piecewise smooth functions with derivative singularities}

We consider
\begin{align}
f_1(x)&=e^x\cos(5x)+\frac{x}{1+x^2},\\
f_2(x)&=f_1(x)+x-\xi,\\
f_3(x)&=f_2(x)+(x-\zeta)^2,
\end{align}
and define
\[
f(x)=
\begin{cases}
f_1(x), & x<\xi,\\
f_2(x), & \xi\le x<\zeta,\\
f_3(x), & x\ge \zeta.
\end{cases}
\]

Here \(f\) is continuous, with a first-derivative discontinuity at \(x=\xi\) and a second-derivative discontinuity at \(x=\zeta\).

We test:
\[
(\xi,\zeta)=(0.2,0.75),\quad (0.225,0.775),\quad (0.21,\pi/4).
\]

\paragraph{Detection.}
We use the indicator
\[
\eta_k=\|c_k^\epsilon\|_2.
\]
As shown in Figure~\ref{fig:piecewise_detection}, \(\eta_k\) remains stable in smooth regions and exhibits clear spikes when a singularity lies inside a subinterval.

\paragraph{Localization.}
For each detected region, we compute one-sided indicators
\[
\eta_{L,i},\qquad \eta_{R,i}.
\]
Their jump behavior identifies the singular location. Coinciding jumps indicate alignment with a sampling node; offset jumps indicate a singularity inside a grid cell.

\paragraph{Correction.}
For each singularity-containing subinterval, we construct two one-sided LLF reconstructions using predicted interface values and define a piecewise approximation.

Figures~\ref{fig:llf_sing_loc1}--\ref{fig:llf_sing_loc2} show that the correction significantly reduces local errors while preserving global accuracy.

\begin{figure}[htbp]
\begin{center}
\subfigure[Indicator \(\eta_k\), Case I: \(\xi=0.2\), \(\zeta=0.75\)]{
\resizebox*{5.2cm}{!}{\includegraphics{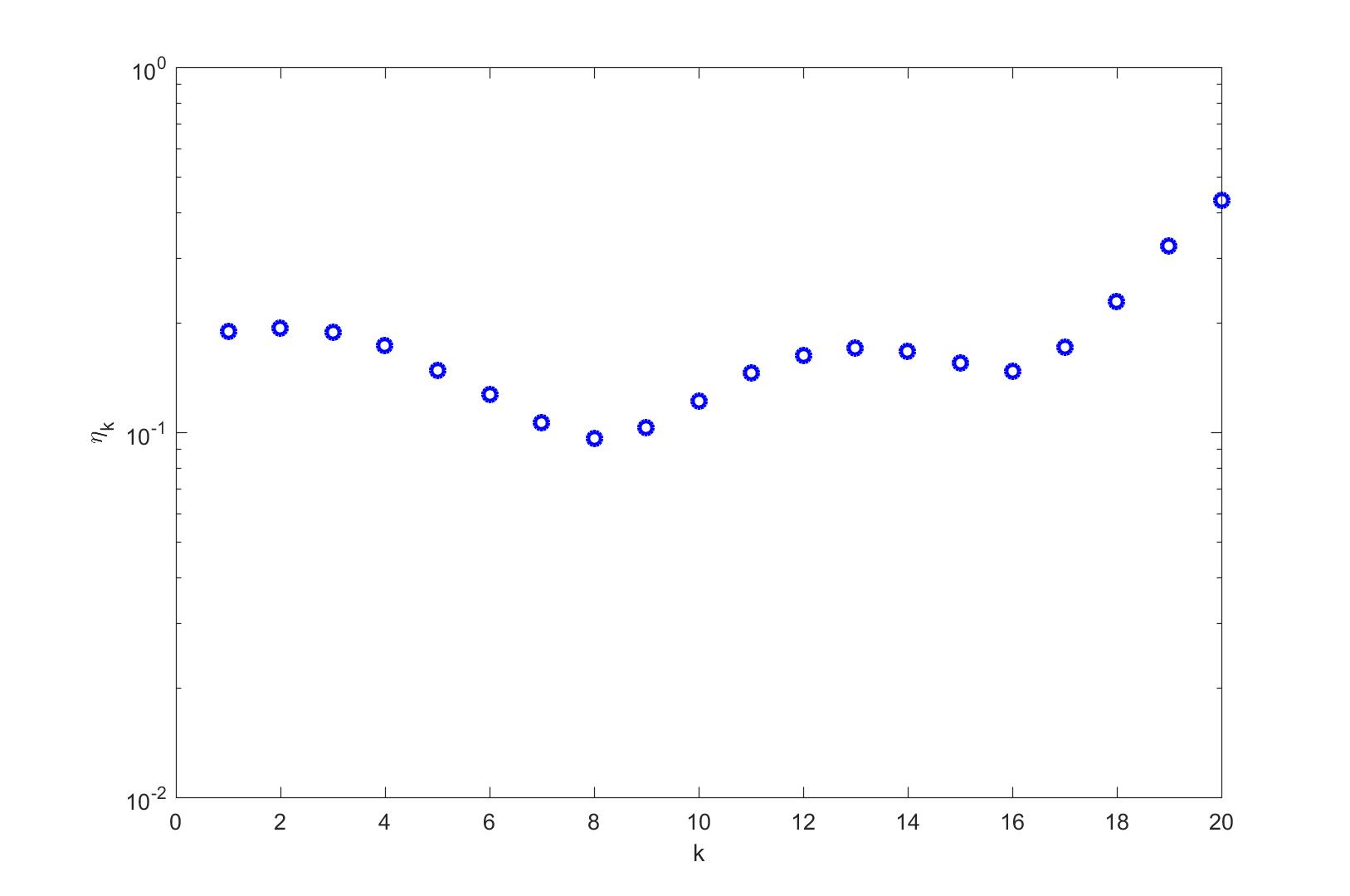}}
}%
\subfigure[Pointwise error, Case I]{
\resizebox*{5.2cm}{!}{\includegraphics{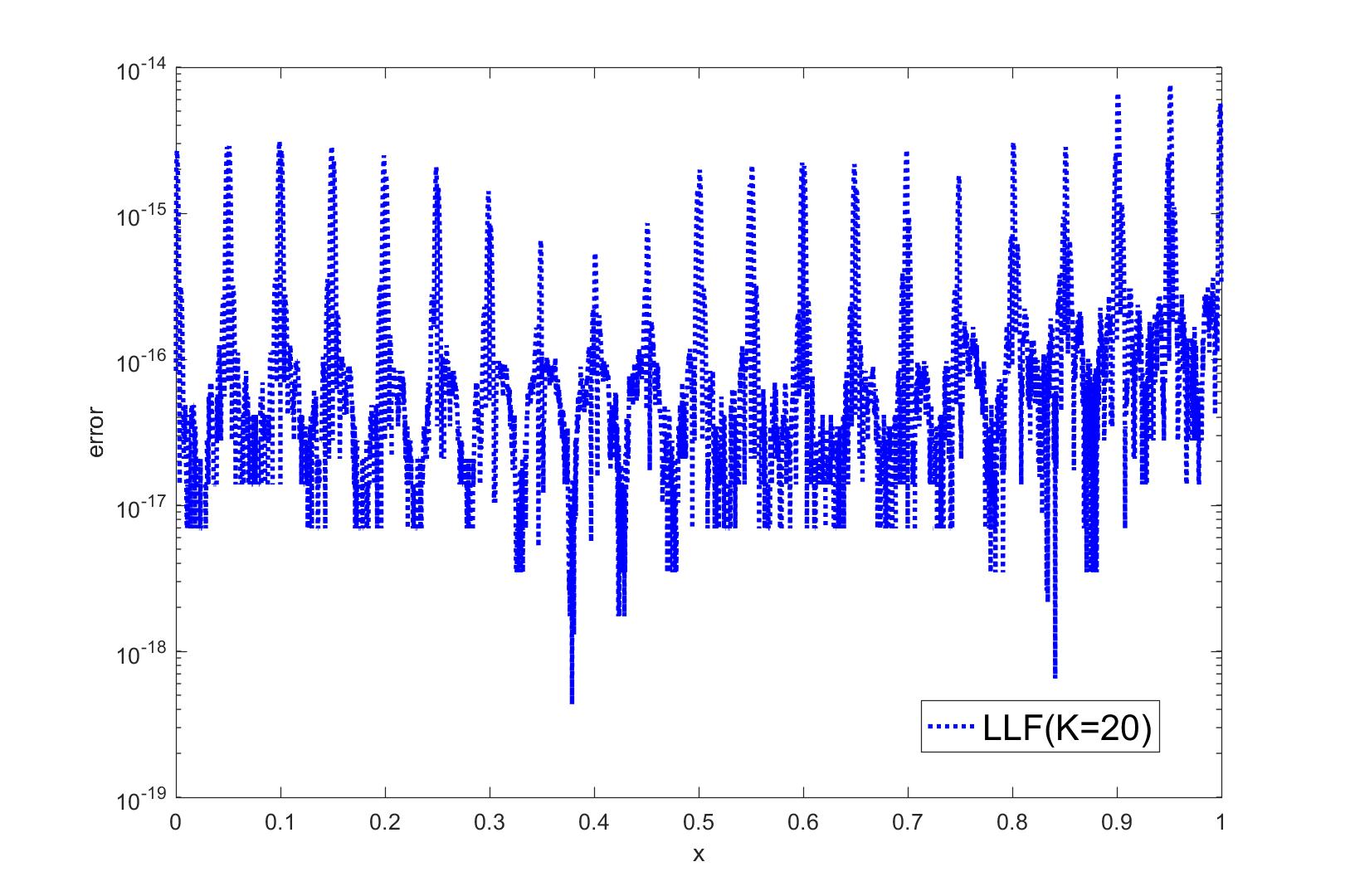}}
}

\subfigure[Indicator \(\eta_k\), Case II: \(\xi=0.225\), \(\zeta=0.775\)]{
\resizebox*{5.2cm}{!}{\includegraphics{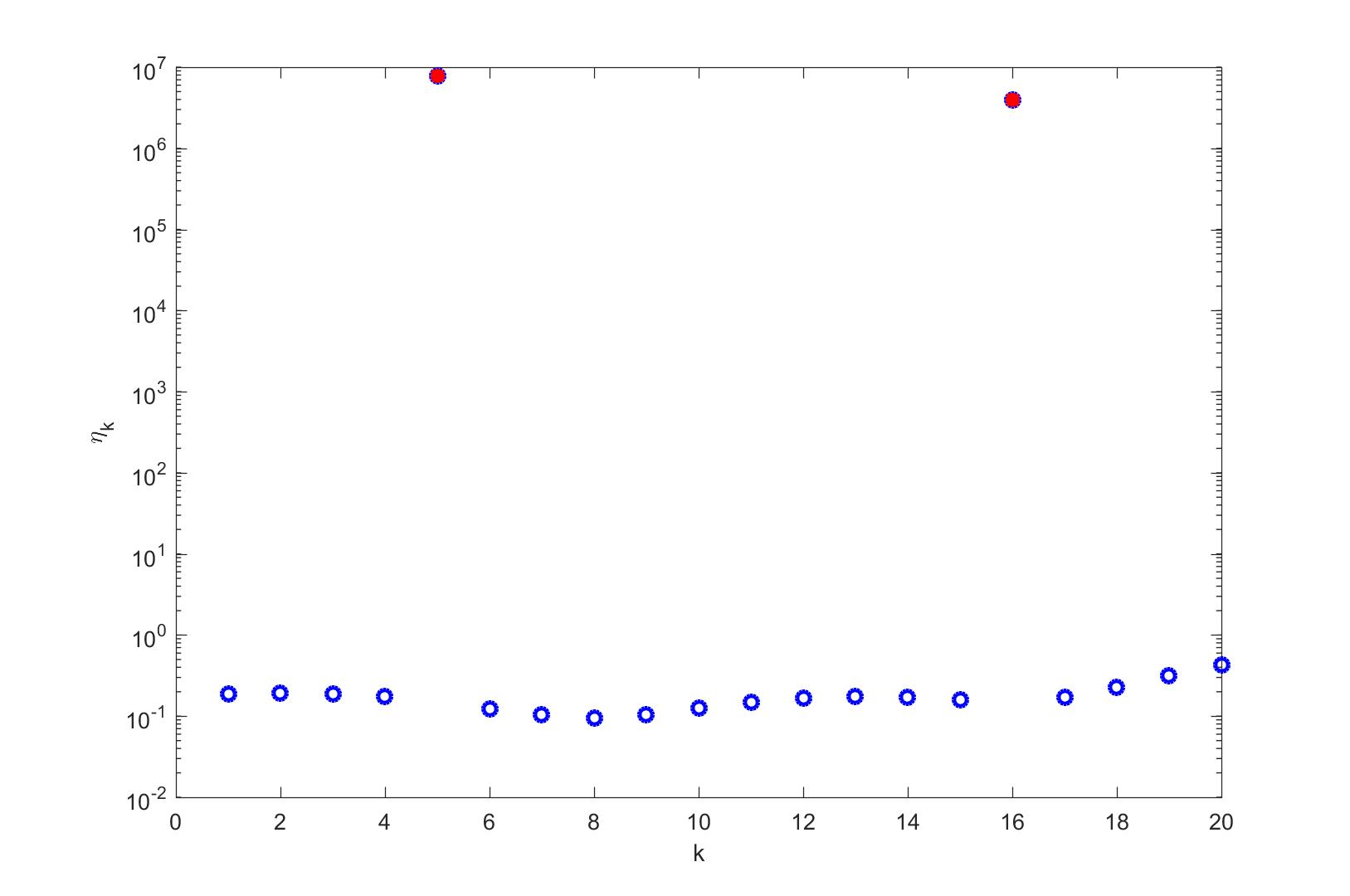}}
}%
\subfigure[Pointwise error, Case II]{
\resizebox*{5.2cm}{!}{\includegraphics{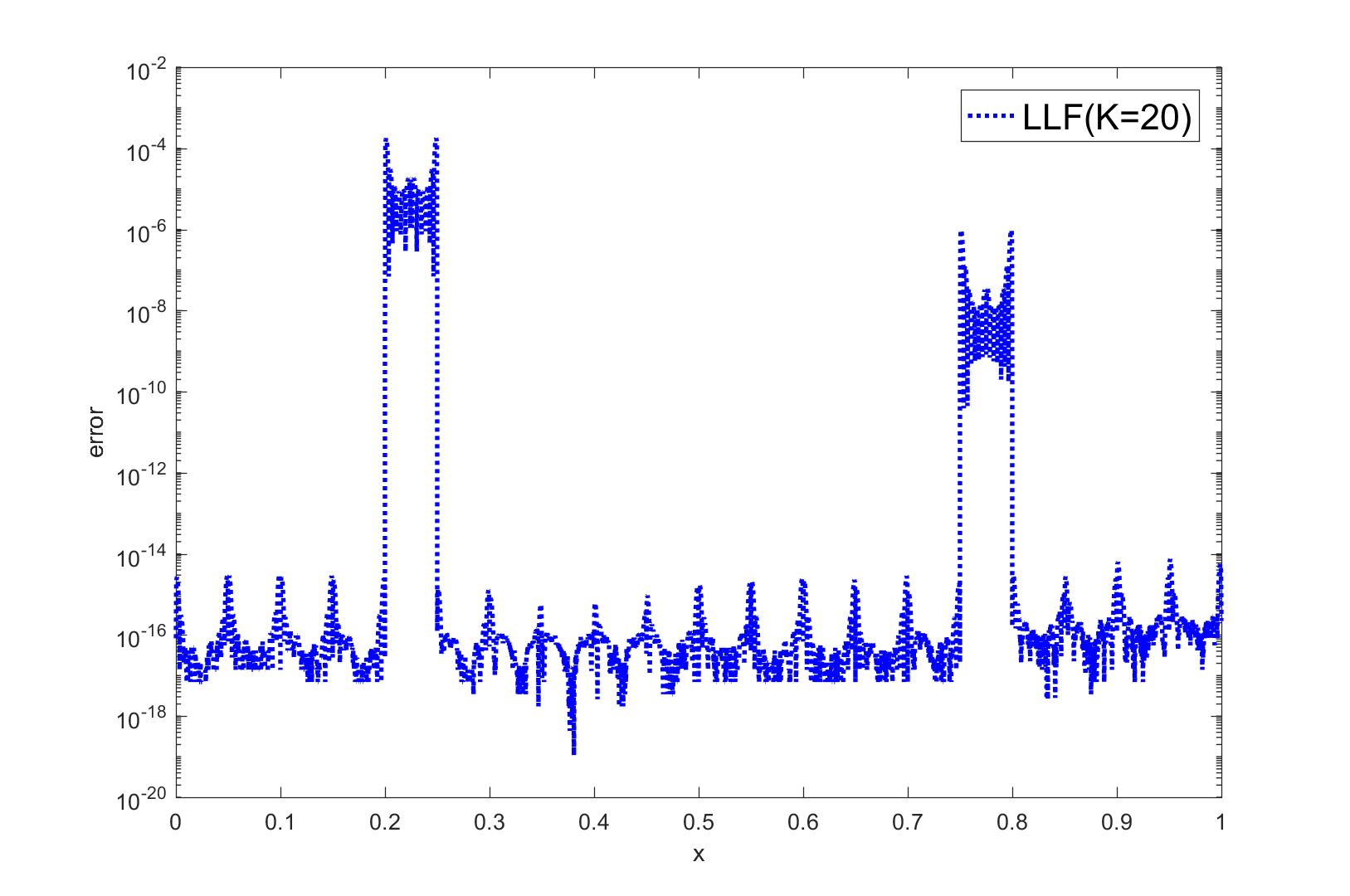}}
}

\subfigure[Indicator \(\eta_k\), Case III: \(\xi=0.21\), \(\zeta=\pi/4\)]{
\resizebox*{5.2cm}{!}{\includegraphics{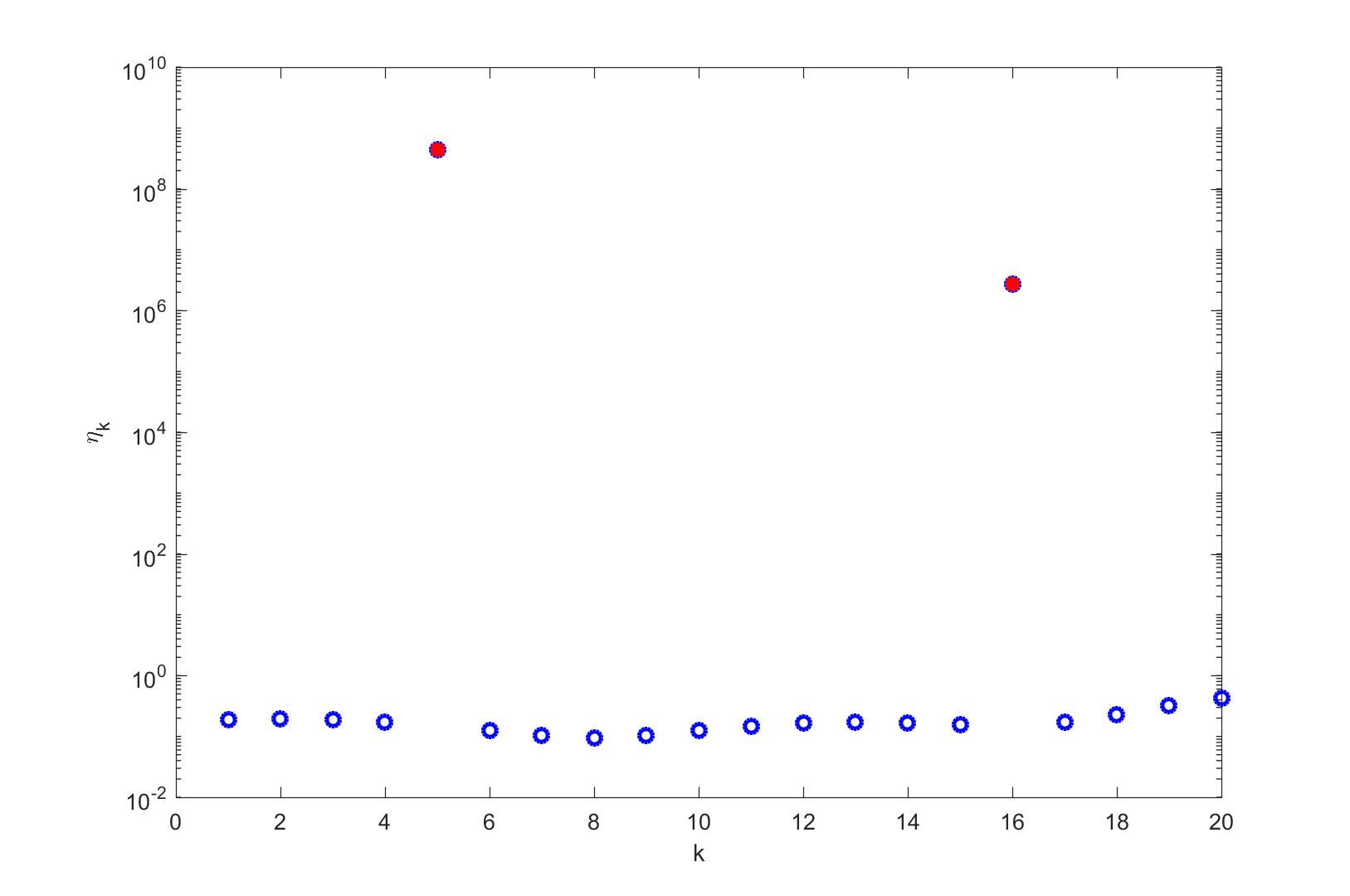}}
}%
\subfigure[Pointwise error, Case III]{
\resizebox*{5.2cm}{!}{\includegraphics{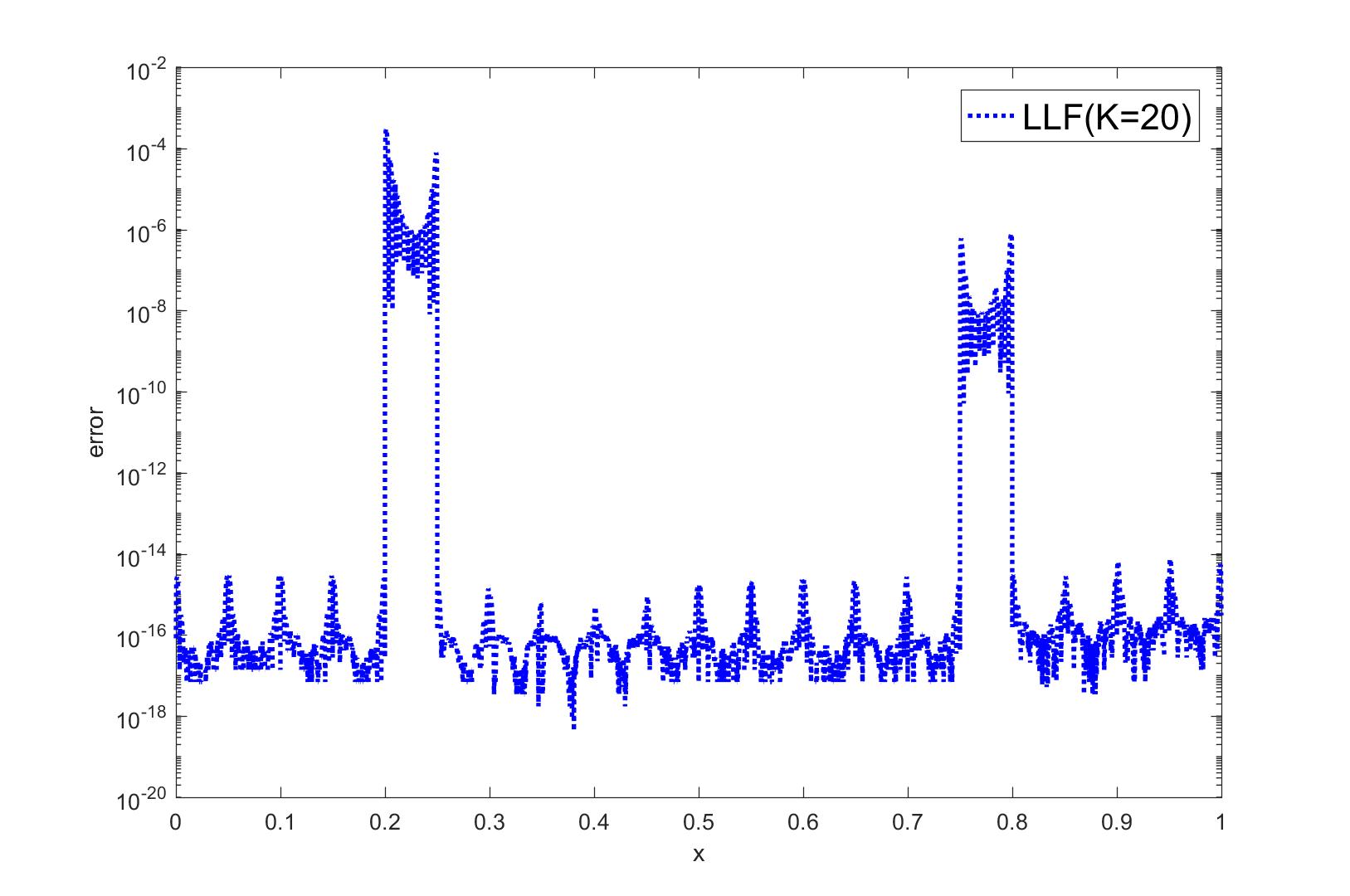}}
}

\caption{Detection of singularity-containing subintervals for the continuous piecewise smooth test function \eqref{eq:piecewise_test}. In Case I, the singular points are located at subinterval endpoints, so no abnormal growth of the indicator is observed and the approximation error remains at the level of machine precision. In Cases II and III, where the singular points lie inside subintervals, the coefficient-energy indicator \(\eta_k\) exhibits clear spikes, and the pointwise error distributions show corresponding local deterioration.}
\label{fig:piecewise_detection}
\end{center}
\end{figure}

\begin{figure}[htbp]
\begin{center}
\subfigure[One-sided localization indicator for the first detected singular window.]{
\resizebox*{5.2cm}{!}{\includegraphics{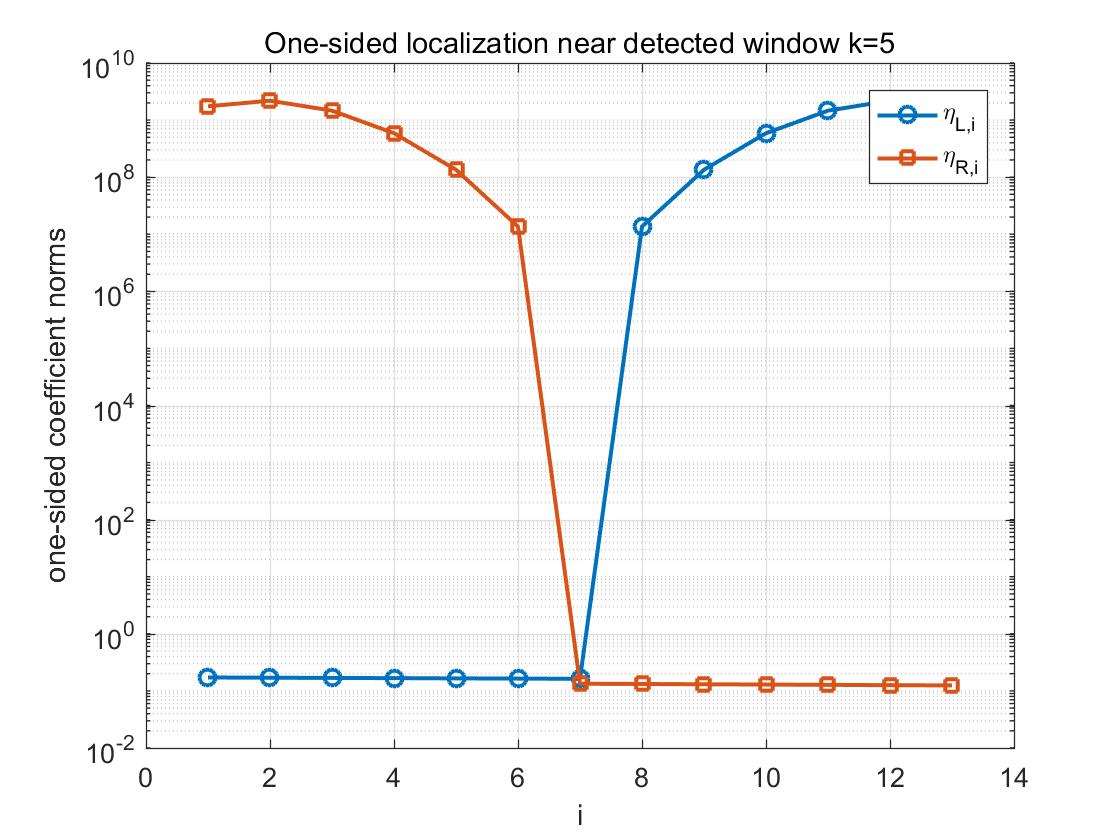}}
}%
\subfigure[One-sided localization indicator for the second detected singular window.]{
\resizebox*{5.2cm}{!}{\includegraphics{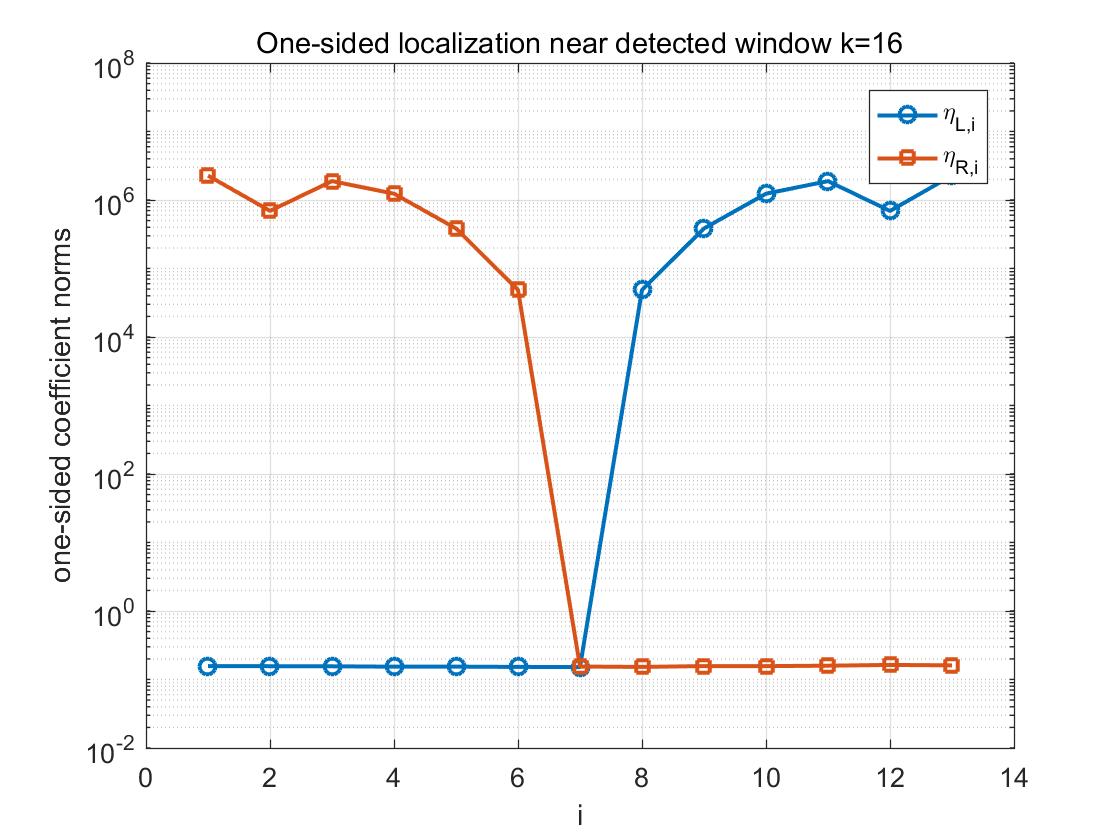}}
}

\subfigure[Pointwise error before and after correcting both detected singular windows.]{
\resizebox*{7.2cm}{!}{\includegraphics{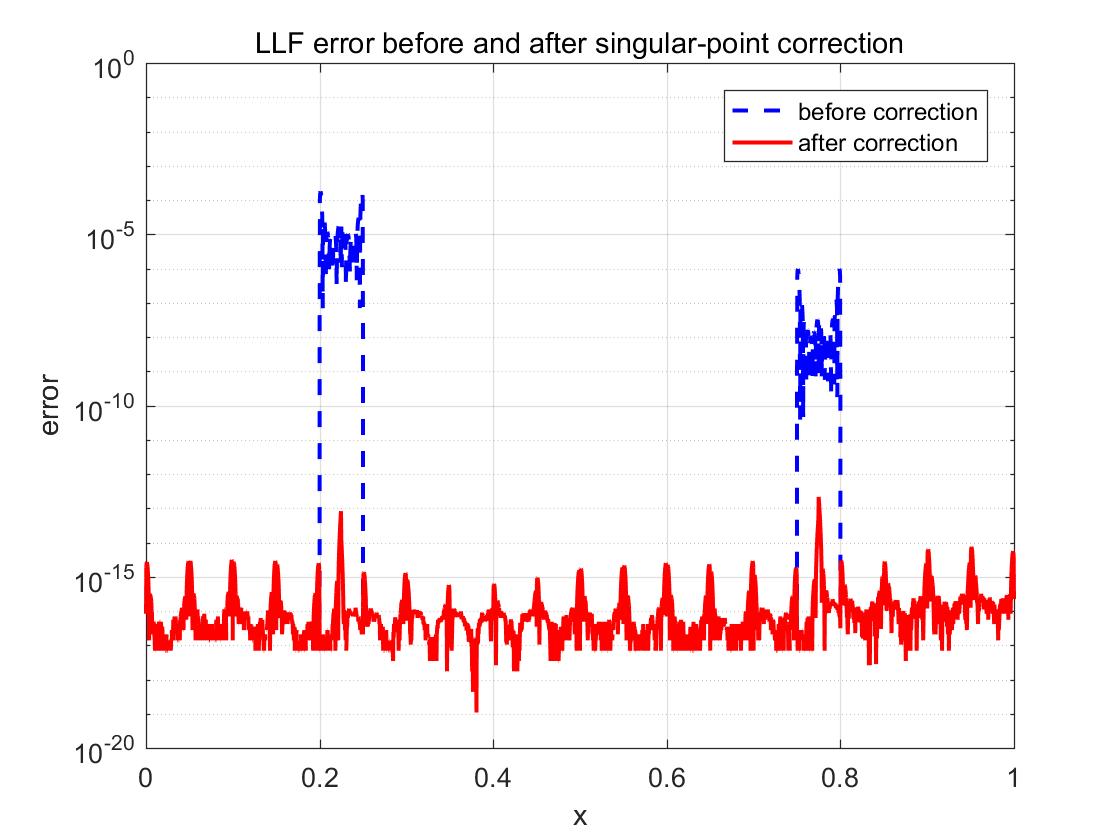}}
}

\caption{Localization and correction for one representative configuration of the piecewise smooth test function. Panels (a) and (b) show the one-sided coefficient-energy indicators \(\eta_{L,i}\) and \(\eta_{R,i}\) for the two detected singularity-containing subintervals. Panel (c) compares the global pointwise errors before and after applying the one-sided correction to both singular windows. The correction substantially reduces the local errors near the singular points while preserving high accuracy on the remaining smooth region.}
\label{fig:llf_sing_loc1}
\end{center}
\end{figure}

\begin{figure}[htbp]
\begin{center}
\subfigure[One-sided localization indicator for the first detected singular window.]{
\resizebox*{5.2cm}{!}{\includegraphics{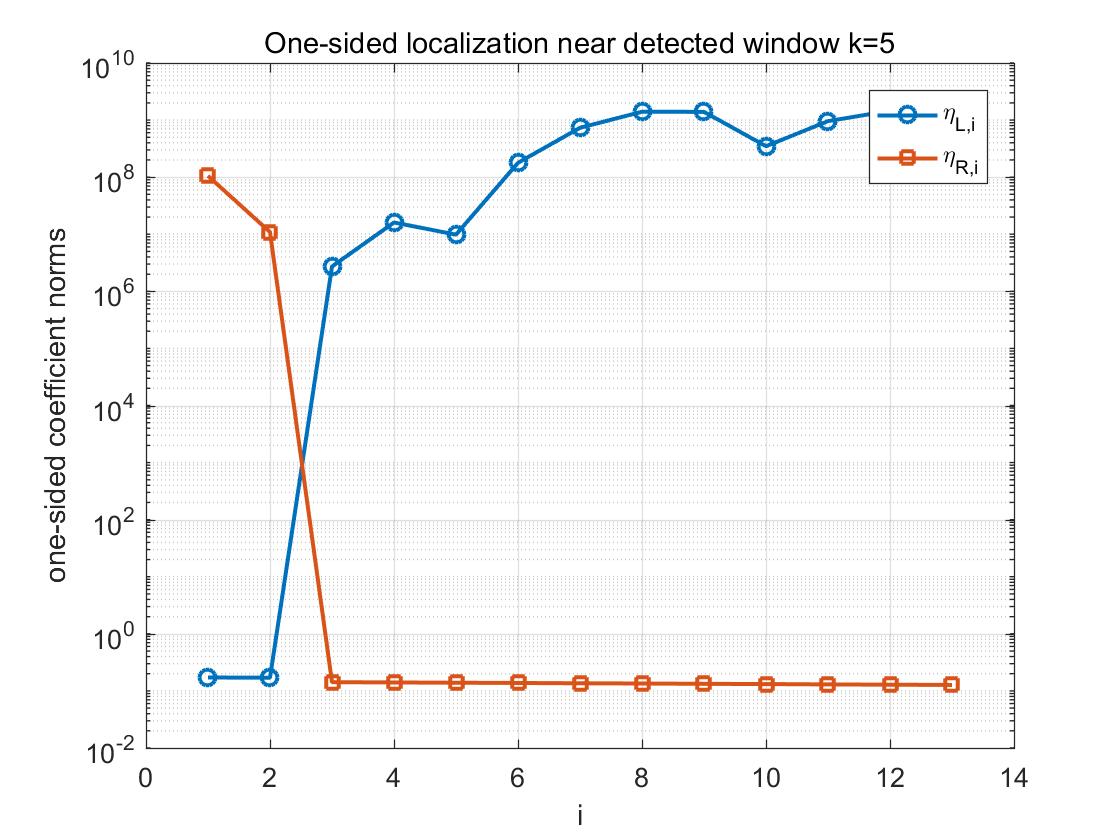}}
}%
\subfigure[One-sided localization indicator for the second detected singular window.]{
\resizebox*{5.2cm}{!}{\includegraphics{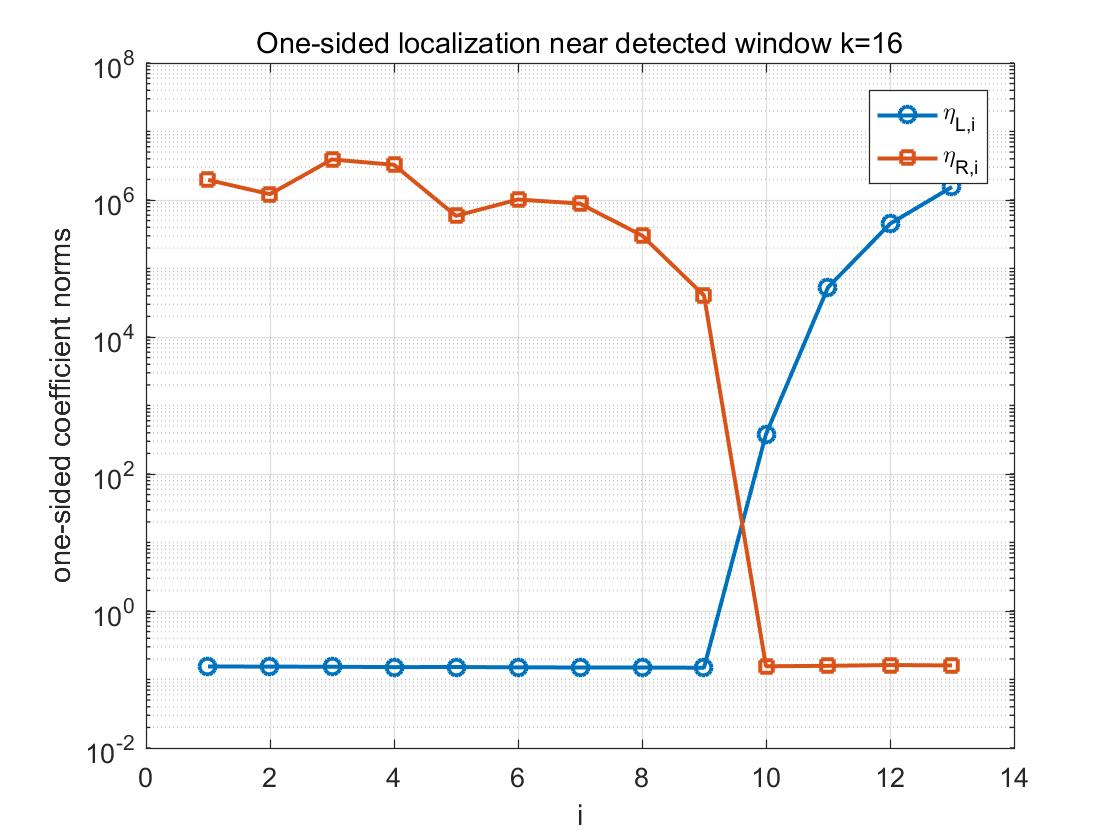}}
}

\subfigure[Pointwise error before and after correcting both detected singular windows.]{
\resizebox*{7.2cm}{!}{\includegraphics{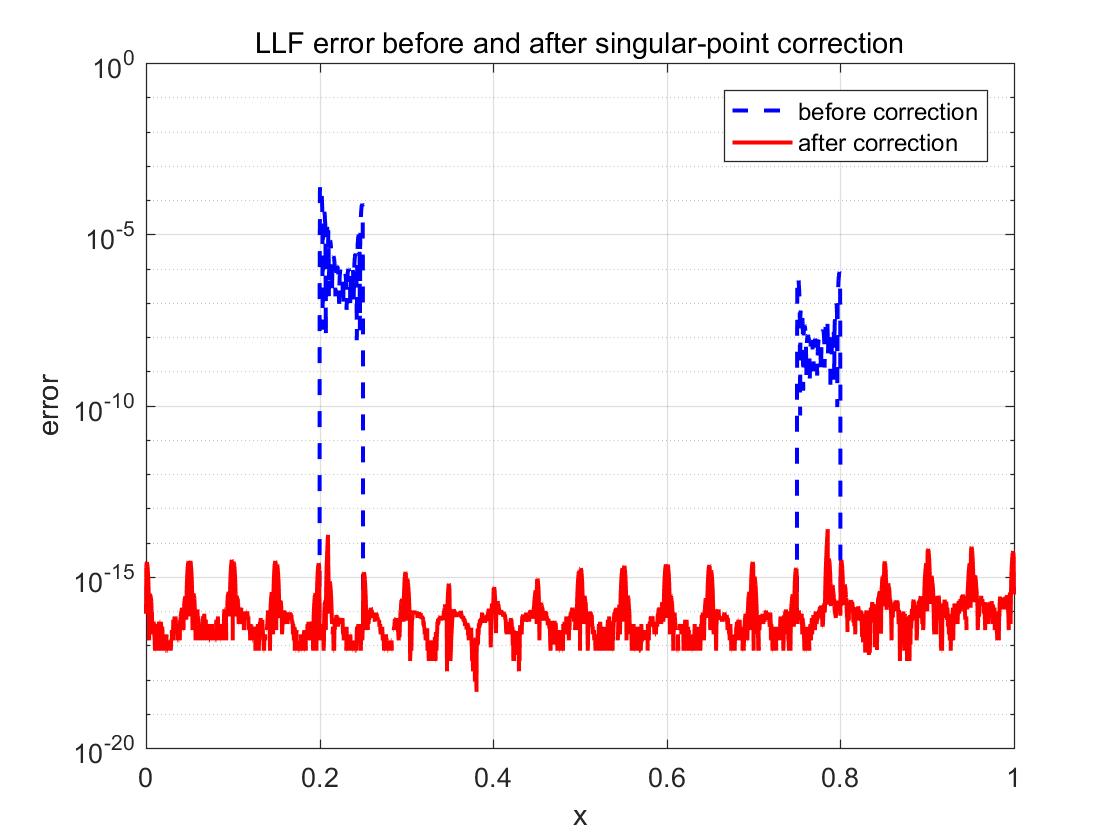}}
}

\caption{A second representative configuration of singularity localization and correction. Panels (a) and (b) again show the one-sided indicators \(\eta_{L,i}\) and \(\eta_{R,i}\) for the two detected singularity-containing subintervals, while panel (c) displays the global pointwise errors before and after correction. The corrected LLF approximation removes the large local errors present in the original approximation and restores high accuracy away from the singular regions.}
\label{fig:llf_sing_loc2}
\end{center}
\end{figure}
\subsection{Discussion}

The experiments reveal a clear regime-dependent behavior:

\begin{itemize}
\item For smooth and moderately oscillatory functions, LLF and LFE achieve comparable accuracy, with LLF slightly cheaper.
\item For highly oscillatory functions, LLF requires significantly more nodes, and LFE is more efficient.
\item For piecewise smooth functions, LLF supports an effective detect--localize--correct strategy.
\end{itemize}

Thus, LLF is a competitive method for smooth and moderately oscillatory data, while LFE remains preferable for strongly oscillatory problems. These observations highlight that LLF is particularly suitable for problems where stability and local adaptivity are more critical than optimal resolution of high-frequency oscillations.

\section{Conclusion}\label{sec:conclusion}

We proposed a local Legendre frame method for approximation from equispaced data. The method combines interval subdivision with TSVD-regularized polynomial-frame approximation and admits an efficient offline--online implementation.

The parameter study shows that \(T\) controls stability, while \(N\) determines resolution. A practical choice is
\[
\gamma=1,\qquad T=6,
\]
with \(N\) selected according to local frequency.

Numerical results show that LLF performs comparably to LFE for smooth and moderately oscillatory functions, but requires more nodes for highly oscillatory problems. For piecewise smooth functions, LLF provides an effective framework for detecting and correcting singularities.

The proposed framework also provides a natural bridge between polynomial approximation and frame-based regularization, suggesting potential extensions toward hybrid approximation strategies. Future work includes adaptive partition strategies, theoretical analysis of numerical rank, and hybrid polynomial--Fourier approaches.

\section{Acknowledgments}
The work was partly supported by Natural Science Foundation of Shandong Province under Grant No. ZR2025MS28.

\end{document}